\documentclass[reqno,twoside]{amsart}
\usepackage{subeqnarray}
\usepackage{mathrsfs}
\usepackage{amscd}
\usepackage{dsfont}
\usepackage{amssymb}
\usepackage{amsmath}
\usepackage{amsfonts}
\usepackage{latexsym}
\usepackage{verbatim}
\usepackage{graphics}
\usepackage{enumerate}
\usepackage{epsfig}
\def\theequation{\thesection.\arabic{equation}}

\theoremstyle{plain}
\newtheorem{theorem}{Theorem}[section]
\newtheorem{definition}[theorem]{Definition}
\newtheorem{lemma}[theorem]{Lemma}
\newtheorem{prop}[theorem]{Proposition}
\newtheorem{cor}[theorem]{Corollary}
\newtheorem{rem}[theorem]{Remark}

\newtheorem*{main}{Main Theorem}

\newcommand{\no}{\nonumber}
\newcommand{\bq}{\begin{equation}}
\newcommand{\eq}{\end{equation}}

\newcommand{\qq}{\qquad}
\newcommand{\q}{\quad}

\newcommand\ds{\displaystyle}

\newcommand{\eps}{\varepsilon}
\newcommand{\e}{\varepsilon}

\def\t{\tau}

\newcommand{\A}{\mathcal A}
\renewcommand{\t}{\tau}

\newcommand\R{{\mathbb R}}

\newcommand\N{{\mathbb N}}
\renewcommand{\phi}{\varphi}

\allowdisplaybreaks

\begin{document}
\title[Rigorous derivation of the Kuramoto-Sivashinsky equation]
{Rigorous derivation of the Kuramoto-Sivashinsky equation in a
$2$D weakly nonlinear Stefan problem.}
\author[C.-M. Brauner]{Claude-Michel Brauner}
\address{C.-M.B.: Institut de Math\'ematiques de Bordeaux\\
Universit\'e de Bordeaux \\
33405 Talence cedex (France)\\} \email{claude-michel.brauner@u-bordeaux1.fr}
\author[J. Hulshof]{Josephus Hulshof}
\address{J.H.: Faculty of Sciences, Mathematics and Computer Sciences Division\\
VU University Amsterdam\\
1081 HV Amsterdam (The Netherlands)} \email{jhulshof@cs.vu.nl}
\author[L. Lorenzi]{Luca Lorenzi$^{\dag}$}\thanks{$\dag$ Corresponding author (luca.lorenzi@unipr.it)}
\address{L.L.: Dipartimento di Matematica\\
Universit\`a di Parma\\
Viale G.P. Usberti 53/A, 43124 Parma (Italy)}
\email{luca.lorenzi@unipr.it}
\urladdr{www.unipr.it/$\sim$lorluc99/index.html}
\begin{abstract}
In this paper we are interested in a rigorous derivation of the
Kuramoto-Sivashinsky equation (K--S) in a Free Boundary Problem. As
a paradigm, we consider a two-dimensional Stefan problem in a strip,
a simplified version of a solid-liquid interface model. Near the
instability threshold, we introduce a small parameter $\e$ and
define rescaled variables accordingly. At fixed $\e$, our method is
based on: definition of a suitable linear $1$D operator, projection
with respect to the longitudinal coordinate only, Lyapunov-Schmidt
method. As a solvability condition, we derive a self-consistent
parabolic equation for the front. We prove that, starting from the
same configuration, the latter remains close to the solution of K--S
on a fixed time interval, uniformly in $\e$ sufficiently small.
\end{abstract}
\maketitle

\section{Introduction}
A very challenging problem in Free Boundary Problems is the
derivation of a single equation for the interface or moving front
which captures the dynamics of the system, at least asymptotically,
when a suitable parameter $\e$ tends to $0$. This program has been
formally achieved by Sivashinsky in the pioneering paper \cite{siva}
within the context of Near-Equidiffusional Flames (NEF) in
combustion theory (see \cite{mat-siva}). Near the instability
threshold, achieved at the critical value $\alpha=1$ ($\alpha$
reflects the physico-chemical characteristics of the combustible),
the {\it dispersion relation} between the wave number $k$ and the
growth rate $\omega_k$ reads:
\begin{eqnarray*}
\omega_k= (\alpha -1)k^2-4k^4,
\end{eqnarray*}
and its counterpart in the physical coordinates is the Kuramoto-Sivashinsky equation
\begin{equation}
\label{eqn:KS} \Phi_{\tau} + \nu\Phi_{\eta\eta\eta\eta} +
\Phi_{\eta\eta} + \frac{1}{2}({\Phi_{\eta})}^2 = 0
\end{equation}
(with $\nu = 4)$, a kind of modulation equation in the rescaled
independent variables $\t=t\e^2$ and $\eta=y\sqrt{\e}$, when the
small parameter $\e= \alpha -1$ tends to $0$. The
Kuramoto-Sivashinsky equation, that we abbreviate hereafter as the
K--S equation, or simply K--S, appears in a variety of domains in
physics and chemistry, where it models cellular instabilities,
pattern formation, turbulence phenomena and transition to chaos, see
(among many other references) \cite{HN,temam} and the bibliography
therein. There are many heuristic derivations of the K--S equation
in the literature.  Our purpose here is to provide some rigorous
mathematical commentary on the derivation of this well-known model.

As one would surmise at the outset, the K--S model comprises a
balance between several effects.  Roughly speaking, K--S arises when
the competing effects of a destabilizing linear part and a
stabilizing nonlinearity are the dominant processes in physical
reality.  The linear instability is itself the result of a
competition between two linear operators, $\A=D_{\eta\eta}$ and
$\nu\A^2$ (we call $\nu\A^2+\A$ the Kuramoto-Sivashinsky linear
operator).

Put another way, the K--S equation is the simplest, and indeed a
paradigm system in which these effects compete equally.  It is this
dominant balance that is explored rigorously in the present essay.
It will turn out that in deriving K--S as an asymptotic limit of
more complex systems, only certain type of terms contribute to the
lowest order of approximation.  Other types of terms will lead to
higher order perturbations. In a forthcoming paper, we intend to
consider the effects of these higher order perturbations on the
basic K--S system.

As a paradigm two-dimensional problem (see \cite{BFHS,BFHR,BFHLS}
for the one-dimensional case and the Q--S equation in flame front
dynamics), we consider a solid-liquid interface model introduced by
Frankel in \cite{frankel}. The solidification front is represented
by $x=\xi(t,y)$. The liquid phase occurs when $x<\xi(t,y)$, the
solid one when $x>\xi(t,y)$. The dynamics of heat is described by
the heat conduction equation
\begin{equation} \label{eqn:T}
T_t(t,x,y)= \Delta T(t,x,y),\quad \; x\neq \xi(t,y),
\end{equation}
where $y \in [-\ell/2, \ell/2]$ with periodic boundary conditions.
At $-\infty$, the temperature of the liquid is normalized to $0$. At
the front $x = \xi(t,y)$ there are two conditions. First, the
balance of energy at the interface is given by the jump
\begin{equation} \label{cond1} \left [\frac{\partial T}{\partial n}\right ] =
V_n,
\end{equation}%
where $V_n$ is the normal velocity.
Second, according to the Gibbs-Thompson law, the non-equilibrium interface temperature is defined by%
\begin{equation} \label{cond2} T=1-\gamma \kappa + r(V_n), \end{equation}%
where the melting temperature has been normalized to $1$, $\kappa$
is the interface curvature and the positive constant $\gamma$
represents the solid-liquid surface tension. The function $r$ is
increasing and such that  $r(-1)=0, \; r'(-1)= 1$, see
\cite{frankel,FRS}. Hereafter, we assume that $r-1$ is linear and we
replace the curvature by the second order derivative. Therefore,
\eqref{cond2} becomes:
\begin{equation}
\label{cond2bis}
T=1-\gamma \xi_{yy} + V_n +1.
\end{equation}%
It is no difficult to see that System \eqref{eqn:T},
\eqref{cond1}, \eqref{cond2} admits a one-phase planar travelling
wave (TW) solution $\hat{T}$, which satisfies
\begin{eqnarray*}
\hat{T}_x= \hat{T}_{xx},\quad\; x\neq 0.
\end{eqnarray*}
At the front $x=0$,
\begin{eqnarray*}
[{\hat{T}}_x] = -1,\quad\; \hat{T}=1.
\end{eqnarray*}
Hence, $ \hat{T}(x) = e^x$ for $x<0$, and $\hat{T}(x)=1$ for
$x>0$.

As usual, we fix the free boundary. We set $\xi(t,y)=-t +
\varphi(t,y)$, $x'=x-\xi(t,y)$ and we will omit primes. In this new
framework, \eqref{eqn:T} reads:
\begin{eqnarray*}
T_t +(1-\varphi_t) T_x = \Delta_\varphi T, \quad x \neq 0,
\end{eqnarray*}
where $\displaystyle \Delta_\varphi = (1 + (\varphi_y)^2)D_{xx} + D_{yy}
-\varphi_{yy}D_x -2\varphi_y D_{xy}.$
The front is now fixed at $x=0$.  The first condition \eqref{cond1}
reads:
\begin{equation} \label{velocity} \varphi_t = 1 + (1+(\varphi_y)^2)\,\left[T_x\right],
\end{equation}
whereas we replace \eqref{cond2bis} by
\begin{eqnarray*}
\label{cond2ter} T=1-\gamma \varphi_{yy} + \varphi_t
+\frac{1}{2}(\varphi_y)^2.
\end{eqnarray*}
Introducing the temperature perturbation $u = T - \hat{T}$, the
problem for the couple $(u,\varphi)$ reads:
\begin{equation}
\label{eqn:u-0} u_t + (1-\varphi_t)u_x -\Delta_\varphi u - \varphi_t
\hat{T}_x = (\Delta_\varphi - \Delta)\hat{T}, \quad x \neq 0,
\end{equation}
where
\begin{eqnarray*}
(\Delta_\varphi - \Delta)\hat{T} =\{(\varphi_y)^2 -
\varphi_{yy}\}e^x\chi_{(-\infty,0)}= \left((\varphi_y)^2 -
\varphi_{yy}\right)\hat{T}_x.
\end{eqnarray*}

As in \cite{BHL08}, we make further simplifications: (i) we consider
a \textit{quasi-steady} problem, dropping the time derivative $u_t$
in \eqref{eqn:u-0}; (ii) we take a linearized problem for $u$; (iii)
we limit ourselves to considering only the second order terms in the
jump conditions at $x=0$. Actually, as it has been observed in
similar problems (see \cite{BFHS}), not far from the instability
threshold the time derivative in the temperature equation has a
relatively small effect on the solution. Our final system reads:
\begin{align}
\label{eqn:u}
&u_x -\Delta u - \varphi_t \hat{T}_x = (\Delta_\varphi -
\Delta)\hat{T}, \quad x \neq 0,\\[1.5mm]
\label{cond-u-1}
&\varphi_t = \left[u_x\right] - (\varphi_y)^2,\\[1.5mm]
\label{cond-u-2}
&u_{|x=0}= -\gamma\varphi_{yy} + \varphi_t
+\frac{1}{2}(\varphi_y)^2.
\end{align}

For the convenience of the reader, we recall the main results of
\cite{BHL08}, where we considered Problem
\eqref{eqn:u}-\eqref{cond-u-2} in the strip $\R \times
[-\ell/2,\ell/2]$, with periodic boundary conditions prescribed at
$y=\pm\ell/2$. More precisely, we studied the stability of the TW
solution and proved the following result: there exists $\gamma_c <1$
such that
\begin{enumerate}[\rm (i)]
\item
for $\gamma>\gamma_c$, the TW solution to Problem
\eqref{eqn:u}-\eqref{cond-u-2} is orbitally stable $($with
asymptotic phase$)$;
\item for
$0<\gamma<\gamma_c$, the TW is unstable.
\end{enumerate}
We also showed that $\gamma_c=1-3\lambda_{1}(\ell)+\cdots$, where
$-\lambda_{1}(\ell)=-4\pi^2/\ell^2$ is the largest eigenvalue of the
realization of $D_{yy}$ in $C([-\ell/2, \ell/2])$
with periodic boundary conditions and zero average.

The main tool is the derivation
of a self-consistent equation for the front $\varphi$:
\begin{equation}
\label{pseudoKS}
\varphi_t+G((\varphi_y)^2)= \Omega\varphi, \qquad |y| \leq \frac{\ell}{2},
\end{equation}
where both $\Omega$ and $G$ are linear pseudo-differential operators whose symbols
$\omega_k$ and $g_k$ are explicit and $g_0=\frac{1}{2}$.
Hence, at the zeroth order $G((\varphi_y)^2)$ coincides with the quadratic term of K--S.
If we think {\it formally} of
\eqref{pseudoKS} in the whole space (i.e. $\ell = +\infty$), then
$\omega_k$ is the growth rate which expands, for small wave number $k$, as
\begin{eqnarray*}
\omega(k)=(1-\gamma)k^2+(\gamma-4)k^4+\cdots,
\end{eqnarray*}
with exchange of stability at $\gamma=1$. Therefore, when $\gamma$
is close to unity, but smaller, it is natural to introduce a small
parameter $\e>0$, setting:
\begin{equation*}
\gamma = 1-\e,
\end{equation*}
and define the rescaled dependent and independent variables accordingly:
\begin{equation}
t= \tau/\eps^2,\quad y = \eta/\sqrt{\eps},\quad u=\eps^2 v,\quad
\varphi = \eps \psi. \label{new-variables}
\end{equation}
Then we anticipate, in the limit $\e \to 0$, that $\psi \simeq \Phi$,
where $\Phi$ solves the following K--S equation (with $\nu=3$):
\begin{equation}
\label{eqn:KS-1}
\Phi_{\tau} + 3\Phi_{\eta\eta\eta\eta} +
\Phi_{\eta\eta} + \frac{1}{2}(\Phi_{\eta})^2 = 0.
\end{equation}

This is what we have to establish in a rigorous mathematical way.
Let us fix $\ell_0>0$. The main idea is to link the small
parameter $\e$ and the width of the strip, which will become larger and larger as $\e \to 0$,
i.e. as $\gamma \to 1$. Take for $\ell$:
\begin{equation*}
\ell_{\e}= \ell_0/\sqrt{\eps},
\end{equation*}
which blows up as $\e \to 0$ and the strip $\R \times
[-\ell_{\e}/2,\ell_{\e}/2]$ approaches $\R^2$. We easily see that
$\lambda_{1}(\ell_{\e})=4\pi^2/\ell^2_{\e} = 4\pi^2\e/\ell_0^2$;
hence,
\begin{eqnarray*}
\gamma_c=1-\frac{12\pi^2}{\ell_0^2}\,\e + \ldots
\end{eqnarray*}
Thus, $\ell_0$ becomes the new bifurcation parameter. We shall
assume that $\ell_0>\sqrt{12}\pi$ in order to have $\gamma_c \in
(1-\e,1)$, i.e., $\gamma>\gamma_c$, otherwise the TW is stable and
the dynamics is trivial. Clearly, this is related to the stability
of the null solution to K--S. The relevant eigenvalue of the
Kuramoto-Sivashinsky linear operator $3\A^2+\A$ is
$3\lambda_{1}(\ell_0)^2-\lambda_{1}(\ell_0)$ which vanishes for
$\lambda_{1}(\ell_0)=1/3$, i.e., when $\ell_0=\sqrt{12}\pi$.

An important feature of this paper is that we work in the {\it fixed
strip} $\R \times [-\ell_0/2,\ell_0/2]$, with the rescaled variables
\eqref{new-variables}. We will return to the original variable only
in the final section.

The main result is the following.

\begin{main}
Let $\Phi_0\in C^{6+2\alpha}([-\ell_0/2,\ell_0/2])$, for some
$\alpha\in (0,1/2)$, satisfy
$D_{\eta}^{(k)}\Phi_0(-\ell_0/2)=D_{\eta}^{(k)}\Phi_0(\ell_0/2)$ for
any $k=0,\ldots, 6$. Let $\Phi$ be the periodic solution of
\eqref{eqn:KS-1} $($with period $\ell_0)$ on a fixed time interval
$[0,T]$, satisfying the initial condition $\Phi(0,\cdot)=\Phi_0$.
Then, there exists $\e_0=\e_0(T)\in (0,1/2)$ such that, for
$0<\e\leq\e_{0}$, Problem \eqref{eqn:u}-\eqref{cond-u-2} admits a
unique smooth solution $(u, \varphi)$ on $[0,T/\e^2]$, which is
periodic with period $\ell_{0}/\sqrt{\varepsilon}$ with respect to
$y$, and satisfies
\begin{eqnarray*}
\varphi(0,y)=\e\Phi_0(y\sqrt{\eps}),\qquad\;\,|y|\le
\frac{\ell_0}{2\sqrt{\e}}.
\end{eqnarray*}
Moreover, there exists a positive constant $C$, independent of
$\e\in (0,\e_0]$, such that
\begin{eqnarray*}
|\varphi(t,y)-\varepsilon
\Phi(t\varepsilon^{2},y\sqrt{\varepsilon})|\leq C\,\varepsilon^{2},\qquad\;\,
0\leq t\leq\frac{T}{\varepsilon^{2}},\;\,|y|\leq\frac{\ell_{0}}{2\sqrt{\varepsilon}}.
\end{eqnarray*}
\end{main}
For a precise definition of what smooth solution means we refer the reader to Section
\ref{subsect-6.3}.

Clearly, the initial condition for $\varphi$ is of special type,
compatible with $\Phi_{0}$\ and \eqref{eqn:KS} at $\tau=0$. Initial
conditions of this type have been already considered in
\cite{BFHLS,BFHS}.

The paper is organized as follows. In Section
\ref{sect-notation-spaces} we introduce some notation and the
function spaces we extensively use throughout the paper. In Section
\ref{sect:ansatz} we proceed to a formal {\it Ansatz} in the spirit
of  \cite{siva}. We set $\gamma=1-\e$, split $v=v^0+\e v^1+\ldots$,
$\psi=\psi^0+\e\psi^1+\ldots$, and show that $\psi^0$ verifies the
K--S equation \eqref{eqn:KS-1}, thanks to an elementary solvability
condition. The paper consists in giving a rigorous proof of the
Ansatz (i.e., to prove the main theorem), thanks to an abstract
solvability condition within the framework of adequate function
spaces. In this respect, in Section \ref{sect:equiv} we transform
System \eqref{eqn:u}-\eqref{cond-u-2} in an equivalent problem (for
the new unknowns) using the techniques of \cite{BHL08}, which are
based on
\begin{enumerate}[\rm (i)]
\item
definition of a suitable linear one-dimensional operator;
\item
projection with respect to the $x$ coordinate only;
\item
Lyapunov-Schmidt method.
\end{enumerate}
This allows us to decouple the system into a self-consistent fourth
order (in space) parabolic equation for the front $\psi$ and an
elliptic equation which can be easily solved whenever a solution to
the front equation is determined. Hence, the rest of the paper is
devoted to study the parabolic equation. In this respect, according
to the Ansatz, we split $\psi=\Phi+\e\rho_{\e}$. In Section
\ref{sect:AP}, we solve the fourth order equation for $\rho_{\e}$,
locally in time, with time domain possibly depending on $\e$. Then,
in Section \ref{sect-6}, we prove that, for any $T>0$, the function
$\rho_{\e}$ exists, and is smooth, in the whole of $[0,T]$ provided
$\e$ is small enough. This result is obtained as a consequence of
some {\em a priori estimates} independent of $\e$, which we prove in
Subsection \ref{sect:AP}. The {\it a priori estimates} are also used
to prove the main theorem (see Subsection \ref{subsect-6.3}).
Finally, some technical tools are deferred to the appendix.

\section{Notation and function spaces}
\setcounter{equation}{0} \label{sect-notation-spaces}
In this
section we introduce some notation and the function spaces which
will be used throughout the paper.

\subsection{Notation}
We denote by $I$, $I_-$ and $I_+$, respectively, the sets
\begin{align*}
&I=\mathbb R\times [-\ell_0/2,\ell_0/2],\\
&I_-=(-\infty,0]\times [-\ell_0/2,\ell_0/2],\\
&I_+=[0,+\infty)\times [-\ell_0/2,\ell_0/2].
\end{align*}
We use the bold notation to denote the elements of both the spaces
$C((-\infty,0])\times C([0,+\infty))$ and $C(I_-)\times C(I_+)$.
Given an element ${\bf u}$ of the previous spaces we denote by $u_1$
and $u_2$ its components. Hence, $u_1\in C((-\infty,0])$ $($resp.
$u_1\in C(I_-)$ and $u_2\in C([0,+\infty))$ $($resp. $u_2\in
C(I_+))$. We write $D_x^{(i)}{\bf u}$ (resp. $D_y^{(i)}{\bf u}$)
$(i=1,2,\ldots)$ to denote the (generalized) function whose
components are $D_x^{(i)}u_1$ and $D_x^{(i)}u_2$ (resp.
$D_y^{(i)}u_1$ and $D_y^{(i)}u_2$).

We extensively use the (generalized) functions ${\bf T}$, ${\bf
T}'$, ${\bf U}$ and ${\bf V}$, which are defined by
\begin{equation}
\left\{
\begin{array}{ll}
T_1(x)=e^x, & x\le 0,\\[2mm]
T_2(x)=1, & x\ge 0,
\end{array}
\right.\qquad\quad\; \left\{
\begin{array}{ll}
T'_1(x)=e^x, & x\le 0,\\[2mm]
T'_2(x)=0, &x\ge 0,
\end{array}
\right.
\label{funct-u2-a}
\end{equation}
\begin{equation}
\left\{
\begin{array}{ll}
U_1(x)=\displaystyle\frac{1-x}{3}e^x, &x\le 0,\\[2mm]
U_2(x)=\displaystyle\frac{1}{3}, &x\ge 0,
\end{array}
\right. \qquad\quad\; \left\{
\begin{array}{ll}
V_1(x)=\ds\left (1-\frac{2}{3}x+\frac{x^2}{6}\right )e^x, & x\le 0,\\[3.5mm]
V_2(x)=\ds 1+\frac{x}{3}, & x\ge 0.
\end{array}
\right. \label{funct-u2-b}
\end{equation}

\subsection{Function spaces}
\label{subsect-funct}
Here, we introduce the function spaces we use
in the paper.

\subsubsection{Spaces of one variable only}
Let us fix $\ell_0>0$ and denote by $L^2$ the space of all square
integrable functions $f$ defined in $(-\ell_0/2,\ell_0/2)$, endowed
with the Euclidean norm
\begin{eqnarray*}
\|w\|_2^2=\int_{-\frac{\ell_0}{2}}^{\frac{\ell_0}{2}}w^2d\eta.
\end{eqnarray*}
Given a real (or even complex valued function) $f\in
L^2(-\ell_0/2,\ell_0/2)$, we denote by $\hat f(k)$ its $k$-th
Fourier coefficient, i.e., we write
\begin{eqnarray*}
f(\eta)=\sum_{k=0}^{+\infty}\hat f(k)w_k(\eta),\qquad\;\,\eta\in
(-\ell_0/2,\ell_0/2),
\end{eqnarray*}
where $\{w_{k}\}$ is a complete set of eigenfunctions of the
operator
\begin{eqnarray*}
A:D(A)={H}^{2}\,\to\, {L}^{2},\qquad\;\,Au=D_{\eta\eta}u,\quad u\in
D(A),
\end{eqnarray*}
with $\ell_0$-periodic boundary conditions, corresponding to the
non-positive eigenvalues
\begin{eqnarray*}
0,-\frac{4\pi^2}{\ell_0^2},-\frac{4\pi^2}{\ell_0^2},-\frac{16\pi^2}{\ell_0^2},
-\frac{16\pi^2}{\ell_0^2},-\frac{36\pi^2}{\ell_0^2},\dots
\end{eqnarray*}
For notational convenience we label this sequence as
\begin{eqnarray*}
0=-\lambda_0>-\lambda_1=-\lambda_2>-\lambda_3=-\lambda_4>\dots
\end{eqnarray*}

\noindent For integer or arbitrary real $s$ we denote by $H^s$ the
usual Sobolev spaces of $\ell_0$-periodic (generalized) functions,
which we conveniently represent as
\begin{equation}
H^{s}=\left\{w=\sum_{k=0}^{+\infty}a_{k}
w_{k}:\,\sum_{k=0}^{+\infty}\lambda_{k}^{s}
a_{k}^{2}<+\infty\right\},
\label{space-H}
\end{equation}
with the usual norm. Next, for any $\beta\ge 0$, we denote by
$C_{\sharp}^{\beta}$ the space of all functions $f\in
C^{\beta}:=C^{\beta}([-\ell_0/2,\ell_0/2])$ such that
$f^{(j)}(-\ell_0/2)=f^{(j)}(\ell_0/2)$ for any $j=0,\ldots,[\beta]$.
The space $C_{\sharp}^{\beta}$ is endowed with the Euclidean norm of
$C^{\beta}([-\ell_0/2,\ell_0/2])$.

\subsubsection{Function spaces of two variables}
Given $h,k\in\N\cup\{0\}$, an interval $J\subset\R$ and a
$($possibly unbounded$)$ closed set $K\subset\R^d$ (for some
$d\in\N$), we denote by $C^{h,k}(J\times K)$, the set of functions
$f:J\times K\to\R$ which are $h$-times continuously differentiable
in $J\times K$ with respect to the first variable and $k$-times
continuously differentiable in $J\times K$ with respect to the
second variable. When $J\times K$ is a compact set, we endow the
space $C^{h,k}(J\times K)$ with the norm
\begin{equation}
\|f\|_{C^{h,k}(J\times K)}= \sup_{s\in
J}\|f(s,\cdot)\|_{C^k(K)}+\sup_{z\in K}\|f(\cdot,z)\|_{C^h(J)},
\label{spazi-h-k}
\end{equation}
for any $f\in C^{h,k}(J\times K)$. Using \eqref{spazi-h-k} we can extend the definition of the spaces $C^{h,k}(J\times K)$ to the case when $h,k\notin\N$.

\noindent Next, we introduce the space ${\mathscr X}$ defined by:
\begin{equation}
{\mathscr X}=\left\{{\bf f}=(f_1,f_2)\in C(I_-)\times C(I_+):~\tilde
f_1\in C_b(I_-), ~\tilde f_2\in C_b(I_+)\right\},
\label{X-space}
\end{equation}
where ``$b$'' stands for bounded and the functions $\tilde f_1$ and $\tilde f_2$ are defined as follows:
\setcounter{equation}{5}
\begin{align*}
&\tilde f_1(x,\eta)=e^{-\frac{x}{2}}f_1(x,\eta),\; \; x\le 0,\;\,|\eta|\le \frac{\ell_0}{2},\\
&\tilde f_2(x,\eta)=e^{-\frac{x}{2}}f_2(x,\eta),\; \; x\ge
0,\;\,|\eta|\le \frac{\ell_0}{2}.
\end{align*}
In the sequel, we will write $\tilde {\bf f}:=(\tilde f_1,\tilde f_2)$.
The space ${\mathscr X}$ is a Banach space when endowed
with the norm
\begin{eqnarray*}
\|{\bf f}\|_{{\mathscr X}}=\|\tilde f_1\|_{C_b(I_-)} +\|\tilde
f_2\|_{C_b(I_+)}:=\sup_{(x,\eta)\in I_-}|\tilde f_1(x,\eta)|+\sup_{(x,\eta)\in I_+}|\tilde
f_1(x,\eta)|,
\end{eqnarray*}
for any ${\bf f}\in {\mathscr X}$.

\section{Formal Ansatz}
\label{sect:ansatz} \setcounter{equation}{0} Let us set
$\gamma=1-\e$ in \eqref{cond-u-2}. Applying the change of variables
defined by \eqref{new-variables} to Problem
\eqref{eqn:u}-\eqref{cond-u-2}, the problem for the couple
$(v,\psi)$ reads (after simplification by $\eps^2$) as follows:
\begin{eqnarray}
 v_x - \left (v_{xx} + \e v_{\eta\eta}\right )=
(\eps\psi_\tau + \eps(\psi_\eta)^2 - \psi_{\eta\eta})\hat{T}_x,
\label{eqn:v}
\end{eqnarray}
and at $x=0$:
\begin{align}
\label{cond-v-1} &\eps  \psi_\tau =\left [v_x\right ] - \eps
(\psi_\eta)^2,\\[1.5mm]
\label{cond-v-2}
&v_{|x=0}=-\psi_{\eta\eta} +  \eps \psi_{\eta\eta}
+\eps (\psi_\tau + \frac{1}{2}(\psi_\eta)^2).
\end{align}
%

In the spirit of \cite[p. 75]{siva}, we look for formal expansions:
\begin{eqnarray*}
v=v^0 + \eps v^1 + \ldots, \qquad \psi = \psi^0 + \eps \psi^1
+\ldots
\end{eqnarray*}
of the solution to Problem \eqref{eqn:v}--\eqref{cond-v-2}.
Considering the zeroth order part of \eqref{eqn:v}--\eqref{cond-v-2}
(i.e., the terms with no powers of $\varepsilon$ in front), it is
easy to see that the function $v^0$ verifies the system
\begin{align}
\label{Otheq1}
&v^0_x - v^0_{xx} = -
\psi^0_{\eta\eta}e^x\chi_{(-\infty,0]},\\[1.5mm]
\label{Otheq2}
&[v^0_x]=0,\\[1.5mm]
\label{Otheq3}
&v^0_{|x=0} = -\psi^0_{\eta\eta}.
\end{align}
It is trivial to solve \eqref{Otheq1} together with e.g.,
\eqref{Otheq3}: it gives
\begin{eqnarray*}
v^0=\left\{
\begin{array}{ll}
-\psi^0_{\eta\eta} e^x (1-x),\q &x\le 0,\\[2mm]
-\psi^0_{\eta\eta}, &x>0.
\end{array}
\right.
\end{eqnarray*}
We remark that \eqref{Otheq2} is automatically verified. Hence, we
are unable to ``close" the nonlinear system for $(v^0, \psi^0)$ at
the zeroth order. This situation is quite common in singular
perturbation theory when the zeroth order can not be fully
determined, see e.g., \cite{eckhaus}. In such a case, one needs to
go to the first order, which is indeed linear. Most often, the
latter demands a solvability condition, for example based on the
Fredholm alternative, which provides the missing relation for the
zeroth order. Therefore, repeating computations similar to the
previous ones, we get the following system for $(v^1, \psi^1)$:
\begin{equation} \label{1steq1} v^1_x - v^1_{xx} -v^0_{\eta\eta} = \{\psi^0_\tau
+ (\psi^0_\eta)^2 - \psi^1_{\eta\eta}\}e^x\chi_{(-\infty,0]}.
\end{equation}
At $x=0$, %
\begin{align}
\label{1steq2}
&[v^1_x] = \psi^0_\tau + (\psi^0_\eta)^2,\\[1.5mm]
\label{1steq3}
&v^1_{|x=0} = - \psi^1_{\eta\eta} + \psi^0_{\eta\eta}
+ \psi^0_\tau + \frac{1}{2}(\psi^0_\eta)^2.
\end{align}
Obviously,
\begin{eqnarray*}
v^0_{\eta\eta} = \left\{
\begin{array}{ll}
-\psi^0_{\eta\eta\eta\eta} e^x
(1-x),\q &x\le 0,\\[2mm]
-\psi^0_{\eta\eta\eta\eta}, & x>0.
\end{array}
\right.
\end{eqnarray*}
Clearly, the solution to \eqref{1steq1} is given by
\begin{eqnarray*}
v^1 = \left\{
\begin{array}{ll}
ae^x+2\psi^0_{\eta\eta\eta\eta} - \psi^0_\tau - (\psi^0_\eta)^2 +
\psi^1_{\eta\eta}xe^x-\frac{1}{2}\psi^0_{\eta\eta\eta\eta}x^2e^x, &
x\le 0,\\[2mm]
a - \psi^0_{\eta\eta\eta\eta} x, & x\ge 0,
\end{array}
\right.
\end{eqnarray*}
where $a$ is an arbitrary parameter. There are two remaining
unknowns at the first order, namely $a$ and $\psi^1_{\eta\eta}$, and
still two relations at $x=0$. First, we use \eqref{1steq3}, which
gives:
\begin{equation} \label{syst1}
a = v^1(0) = - \psi^1_{\eta\eta} + \psi^0_{\eta\eta} + \psi^0_\tau +
\frac{1}{2}(\psi^0_\eta)^2.
\end{equation}
Second, we compute:
\begin{eqnarray*}
v^1_x(0^+)= -\psi^0_{\eta\eta\eta\eta}
\end{eqnarray*}
and
\begin{eqnarray*}
v^1_x(0^-) = a + 2\psi^0_{\eta\eta\eta\eta} - \psi^0_\tau -
(\psi^0_\eta)^2 + \psi^1_{\eta\eta}.
\end{eqnarray*}
Therefore, from \eqref{1steq2} we get:
\begin{equation} \label{syst2} v^1_x(0^+)-v^1_x(0^-) = -a -
\{3\psi^0_{\eta\eta\eta\eta} - \psi^0_\tau - (\psi^0_\eta)^2 +
\psi^1_{\eta\eta} \} = \psi^0_\tau + (\psi^0_\eta)^2. \end{equation}
Obviously \eqref{syst1}-\eqref{syst2} is a linear system for
$(a,\psi^1_{\eta\eta})$ with solvability condition:
\begin{eqnarray*}
\psi^0_{\eta\eta} + \psi^0_\tau +\frac{1}{2} (\psi^0_\eta)^2 +
3\psi^0_ {\eta\eta\eta\eta}=0,
\end{eqnarray*}
i.e., $\psi^0$ verifies a K--S equation.
\section{An equivalent problem to \eqref{eqn:v}--\eqref{cond-v-2}}
\label{sect:equiv} \setcounter{equation}{0}

The aim of this section consists in transforming Problem
\eqref{eqn:v}--\eqref{cond-v-2} into an equivalent one.
More precisely, we are going to decouple the problem for $(v,\psi)$, getting a self-consistent
equation for the front $\psi$ and an equation for the other unknown (say ${\bf z}$) which can be
immediately solved once $\psi$ is known.

In deriving the equivalent problem, we assume that the solution
$(v,\psi)$ to Problem \eqref{eqn:v}--\eqref{cond-v-2} in the time
domain $[0,T]$ belongs to the space ${\mathscr V}_T\times {\mathscr
Y}_T$ where

\begin{definition}
\label{def-YT}
For any $T>0$, we denote by ${\mathscr V}_T$ the
space of all functions $v:[0,T]\times \R\times
[-\ell_0/2,\ell_0/2]\to\R$ such that
\begin{enumerate}[\rm (i)]
\item
$v$ is twice continuously differentiable with respect to the spatial
variable in $[0,T]\times I_-$ and in $[0,T]\times I_+$;
\item
the functions $(\tau,x,\eta)\mapsto
e^{-\frac{x}{2}}D^{(i)}_xu(\tau,x,\eta)$ and $(\tau,x,\eta)\mapsto
e^{-\frac{x}{2}}D^{(i)}_{\eta}u(\tau,x,\eta)$ are bounded in
$[0,T]\times I_-$ and in $[0,T]\times I_+$ for any $i=0,1,2$.
\end{enumerate}
Further, for any $\alpha\in (0,1/2)$, we denote by ${\mathscr Y}_T$
the space of all functions $\zeta\in C^{1,4}([0,T]\times
[-\ell_0/2,\ell_0/2])$, such that $\zeta_{\tau}\in
C^{0,2+\alpha}([0,T]\times [-\ell_0/2,\ell_0/2])$ and
$D_{\eta}^{(j)}\zeta(\cdot,-\ell_0/2)=D_{\eta}^{(j)}\zeta(\cdot,\ell_0/2)$
for $j=0,1,2,3$.
\end{definition}

\begin{rem}
\label{rem-smooth} It is immediate to check that, if $\zeta\in
{\mathscr Y}_T$, then the function $\psi_{\eta\eta}$ is
continuously differentiable in $[0,T]\times [-\ell_0/2,\ell_0/2]$
with respect to $\tau$ and, consequently,
$\psi_{\tau\eta\eta}=\psi_{\eta\eta\tau}$. Hence, in what follows,
we always write $\psi_{\tau\eta\eta}$ instead of
$\psi_{\eta\eta\tau}$.
\end{rem}

\subsection{Derivation of a self-consistent equation for the front}
\label{sect-self-cons}

In this subsection we derive a self-consistent equation for the
front. Since its derivation is rather long, we split the proof into
several steps.

\subsubsection{Elimination of $\psi_{\tau}$}
\label{sub-3.1}
First we eliminate $\psi_{\tau}$ in \eqref{eqn:v}
thanks to \eqref{cond-v-2}, getting the equation
\begin{equation}
v_x -v_{xx}-\varepsilon v_{\eta\eta}-v(\cdot,0,\cdot)\hat{T}_x =
\left
(\frac{1}{2}\varepsilon(\psi_{\eta})^2-\varepsilon\psi_{\eta\eta}\right
)\hat{T}_x. \label{part-u}
\end{equation}
Let us set ${\bf
v}(\tau,x,\eta):=(v(\tau,x,\eta)\chi_{(-\infty,0]}(x),v(\tau,x,\eta)\chi_{[0,+\infty)}(x))$
and
\begin{eqnarray*}
{\bf F}_0=\left
(\psi_{\eta\eta}-\frac{1}{2}(\psi_{\eta})^2\right
){\bf T}', \quad g=\psi_{\tau}+(\psi_{\eta})^2,
\end{eqnarray*}
where ${\bf T}'$ is given by \eqref{funct-u2-a}. Taking
\eqref{velocity} and \eqref{part-u} into account, one can easily
show that the function ${\bf v}$ solves the problem
\begin{equation}
\left\{
\begin{array}{ll}
{\mathscr L}{\bf v}=\e{\bf F}_0 -\varepsilon {\bf v}_{\eta\eta},\\[2mm]
v_2(\cdot, 0,\cdot)-v_1(\cdot,0,\cdot)=0,\\[2mm]
D_xv_2(\cdot,0,\cdot)-D_xv_1(\cdot,0,\cdot)=\varepsilon g,
\end{array}
\right. \label{pb-fin-u}
\end{equation}
where
\begin{eqnarray*}
({\mathscr L}{\bf v})(\cdot,x,\eta)= \left\{
\begin{array}{lll}
D_{xx}v_1(\cdot,x,\eta)-D_xv_1(\cdot,x,\eta)+e^xv_1(\cdot,0,\eta), &x\le 0, &|\eta|\le \frac{\ell_0}{2},\\[2mm]
D_{xx}v_2(\cdot,x,\eta)-D_xv_2(\cdot,x,\eta), &x\ge 0, &|\eta|\le
\frac{\ell_0}{2}.
\end{array}
\right.
\end{eqnarray*}

\subsubsection{Lifting up the boundary conditions}
\label{sub-3.2}
Now we are going to use the first part of
\eqref{velocity}. We introduce the new unknown ${\bf w}={\bf
v}-\varepsilon {\mathscr N}(g)$, where ${\mathscr N}(g)=g({\bf
V}-{\bf T})$, and ${\bf V}$ and ${\bf T}$ are defined in
\eqref{funct-u2-a} and \eqref{funct-u2-b}. With a straightforward
computation, we see that the function ${\bf w}$ turns out to solve
the problem
\begin{equation}
\left\{
\begin{array}{ll}
{\mathscr L}{\bf w}=\e{\bf F}_0 -\varepsilon{\bf
w}_{\eta\eta}-\varepsilon^2 g_{\eta\eta}{\mathscr N}(1)-
\varepsilon g{\mathscr L}{\mathscr N}(1),\\[2mm]
w_2(\cdot,0,\cdot)-w_1(\cdot,0,\cdot)=0,\\[2mm]
D_xw_2(\cdot,0,\cdot)-D_xw_1(\cdot,0,\cdot)=0.
\end{array}
\right. \label{Lw=0}
\end{equation}
Since ${\bf v}\in {\mathscr V}_T$, ${\bf v}(\tau,\cdot)$, ${\mathscr
L}{\bf v}(\tau,\cdot) \in {\mathscr X}$ (see \eqref{X-space} for the
definition of the space ${\mathscr X}$) for any $\tau\in [0,T]$,
then a straightforward computation shows that the function ${\bf
w}(\tau,\cdot)$ belongs to ${\mathscr X}$ for any $\tau\in [0,T]$,
and, hence, to the set
\begin{eqnarray*}
\Big\{{\bf h}\in C^{2,0}(I_-)\times C^{2,0}(I_+): {\bf h},~{\mathscr
L}{\bf h}\in {\mathscr
X},\;D_x^{(j)}h_1(0,\cdot)=D_x^{(j)}h_2(0,\cdot),~j=0,1\Big\},
\end{eqnarray*}
which is the domain of the realization $L$ of the operator
${\mathscr L}$ in ${\mathscr X}$, see Section \ref{Lstuff}.

\subsubsection{A Lyapunov-Schmidt method}
\label{sub-3.3} From the results in the previous subsection, we know
that ${\bf w}(\tau,\cdot)\in D(L)$ for any $\tau\in [0,T]$, and it
solves the equation
\begin{equation}
{L}{\bf w}=\varepsilon{\bf F}_0 -\varepsilon{\bf
w}_{\eta\eta}-\varepsilon^2g_{\eta\eta}{\mathscr N}(1)-\varepsilon
g{\mathscr L}{\mathscr N}(1). \label{eq-w}
\end{equation}
We are going to project \eqref{eq-w} along a suitable subspace of
${\mathscr X}$, to derive a self-consistent equation for the front
$\psi$.

As Theorem \ref{thm:1} shows, the operator $L$ is sectorial in ${\mathscr X}$.
Hence, it generates an analytic semigroup. Moreover, $0$ is an isolated eigenvalue
of $L$ and the spectral projection on the kernel of $L$ is the operator ${\mathscr P}$
defined by
\begin{eqnarray*}
{\mathscr P}({\bf f})=\left
(\int_{-\infty}^0f_1(x,\cdot)dx+\int_0^{+\infty}e^{-x}f_2(x,\cdot)dx\right
) {\bf U}:= Q({\bf f}){\bf U},\qquad\;\,{\bf f}\in {\mathscr X}.
\end{eqnarray*}
From the very general theory of analytic semigroup, it follows that,
for a given ${\mathbf g}\in {\mathscr X}$, the equation $L{\bf
z}={\bf g}$ admits a solution ${\bf z}\in D(L)$ if and only if
${\mathscr P}({\bf g})=0$. Since ${\bf w}$ solves Equation
\eqref{eq-w}, it follows that
\begin{eqnarray*}
{\mathscr P}(\e {\bf F}_0 -\varepsilon {\bf w}_{\eta\eta}
-\varepsilon g{\mathscr L}{\mathscr N}(1)-\varepsilon^2g_{\eta\eta}{\mathscr N}(1))=0,
\end{eqnarray*}
or equivalently, after division by $\e>0$,
\begin{equation}
0=Q({\bf F}_0 -{\bf w}_{\eta\eta}-g{\mathscr
L}{\mathscr N}(1)-\varepsilon g_{\eta\eta}{\mathscr N}(1) ). \label{Q=0}
\end{equation}
Since \setcounter{equation}{6}
\begin{align}
&Q({\bf F}_0)= \psi_{\eta\eta}-\frac{1}{2}(\psi_{\eta})^2,\tag{\theequation a}\label{Q-1}\\
&Q(g_{\eta\eta}{\mathscr N}(1))= \frac{4}{3}g_{\eta\eta}= \frac{4}{3}\left(\psi_{\tau\eta\eta}+((\psi_{\eta})^2)_{\eta\eta}\right),\tag{\theequation b}\label{Q-2}\\
&Q(g{\mathscr L}{\mathscr
N}(1))=-g=-\psi_{\tau}-(\psi_{\eta})^2,\tag{\theequation
c}\label{Q-3}
\end{align}
we can rewrite Equation \eqref{Q=0} as follows:
\begin{eqnarray}
\psi_{\tau}
-\frac{4}{3}\varepsilon\psi_{\tau\eta\eta}+\frac{1}{2}(\psi_{\eta})^2
+\psi_{\eta\eta} -
\frac{4}{3}\varepsilon((\psi_{\eta})^2)_{\eta\eta}= Q({\bf
w}_{\eta\eta}). \label{eq-varphi}
\end{eqnarray}

To get a self-contained equation for the front $\psi$, we have to
give a representation of $Q({\bf w}_{\eta\eta})$ in the right-hand
side of \eqref{eq-varphi}. For this purpose, in the spirit of the
Lyapunov-Schmidt method, we split ${\bf w}(\tau,\cdot)$ ($\tau\in
[0,T]$) along ${\mathscr P}({\mathscr X})$ and $(I-{\mathscr
P})({\mathscr X})$. Writing
\begin{eqnarray*}
{\bf w}=a{\bf U}+\varepsilon{\bf z},
\end{eqnarray*}
and observing that our assumptions on $v$
guarantee that the function ${\bf z}_{\eta\eta}$ belongs to
$(I-{\mathscr P})({\mathscr X})$, we get
\begin{equation}
Q({\bf w}_{\eta\eta})=Q(a_{\eta\eta}{\bf U}+\varepsilon {\bf
z}_{\eta\eta})=a_{\eta\eta}.
\label{Q-w}
\end{equation}
Let us compute $a$ and its derivatives. We use the
relation in \eqref{cond-v-2} to obtain
\begin{eqnarray*}
\frac{1}{3}a+\varepsilon z_1(\cdot,0,\cdot)=(\varepsilon-1)\psi_{\eta\eta}+\varepsilon\psi_{\tau}+\frac{1}{2}\varepsilon(\psi_{\eta})^2.
\end{eqnarray*}
Thus,
\begin{equation}
a_{\eta\eta}=-3\varepsilon D_{\eta\eta}z_1(\cdot,0,\cdot)+3(\varepsilon-1)\psi_{\eta\eta\eta\eta}
+3\varepsilon\psi_{\tau\eta\eta}+\frac{3}{2}\varepsilon((\psi_{\eta})^2)_{\eta\eta}.
\label{eq-A}
\end{equation}
From \eqref{Q-w} and \eqref{eq-A}, it follows that
\begin{eqnarray*}
Q({\bf
w}_{\eta\eta})=-3\varepsilon D_{\eta\eta}z_1(\cdot,0,\cdot)+3(\varepsilon-1)\psi_{\eta\eta\eta\eta}+3\varepsilon\psi_{\tau\eta\eta}
+\frac{3}{2}\varepsilon((\psi_{\eta})^2)_{\eta\eta}.
\end{eqnarray*}
Replacing into \eqref{eq-varphi} we get the following equation for
$\psi$:
\begin{equation}
\psi_{\tau}
-\frac{13}{3}\varepsilon\psi_{\tau\eta\eta}+3(1-\varepsilon)\psi_{\eta\eta\eta\eta}
+\psi_{\eta\eta}+\frac{1}{2}(\psi_{\eta})^2 +
3\varepsilon D_{\eta\eta}z_1(\cdot,0,\cdot)=
\frac{17}{6}\varepsilon((\psi_{\eta})^2)_{\eta\eta}.
\label{parabolic-4order}
\end{equation}
We already see that \eqref{parabolic-4order} reduces to K--S if $\e=0$. However,
we still have $z_1$ in the right-hand side of
\eqref{parabolic-4order}. In the next subsection, we write it in terms of $\psi$.

\subsubsection{The equation for ${\bf z}$}
\label{sub-3.4}
To write $D_{\eta\eta}z_1(\cdot,0,\cdot)$ in terms of the function $\psi$, we determine the
equation satisfied by function ${\bf z}$. Projecting Equation
\eqref{eq-w} along $(I-{\mathscr P})({\mathscr X})$, we see that the
function ${\bf z}(\tau,\cdot)=(I-{\mathscr P}){\bf z}(\tau,\cdot)\in
D(L)$ ($\tau\in [0,T]$) solves the equation
\begin{equation}
L{\bf z}=(I-{\mathscr P})({\bf F}_0) -g(I-{\mathscr P})({\mathscr
L}{\mathscr N}(1))-\varepsilon g_{\eta\eta}(I-{\mathscr
P})({\mathscr N}(1))-\varepsilon{\bf z}_{\eta\eta}. \label{pb-w}
\end{equation}
 From \eqref{Q-1}-\eqref{Q-3} we obtain
\begin{eqnarray*}
&& (I-{\mathscr P})({\bf F}_0)=\left
(\psi_{\eta\eta}-\frac{1}{2}(\psi_{\eta})^2\right )
({\bf T}'-{\bf U}),\no\\[1mm]
&& g_{\eta\eta}(I-{\mathscr P})({\mathscr N}(1))= \left
(\psi_{\tau\eta\eta}+((\psi_{\eta})^2)_{\eta\eta}\right )
\left ({\bf V}-{\bf T}-\frac{4}{3}{\bf U}\right ),\no\\[1mm]
&& g(I-{\mathscr P})({\mathscr L}{\mathscr N}(1))=0,
\end{eqnarray*}
so that we can rewrite Equation \eqref{pb-w} as
\begin{align}
L{\bf z}+\e {\bf z}_{\eta\eta}=&\left
(\psi_{\eta\eta}-\frac{1}{2}(\psi_{\eta})^2\right )({\bf
T}'-{\bf U})\notag\\
&-\varepsilon\left
(\psi_{\tau\eta\eta}+((\psi_{\eta})^2)_{\eta\eta}\right
)\left ({\bf V}-{\bf T}-\frac{4}{3}{\bf U}\right ).
\label{final-w}
\end{align}

We now observe that the operator $L+\e A:=L+\e D_{\eta\eta}$ with domain
\begin{align}
D(L+\eps A)
=&\{{\bf u}\in D(L): {\bf u}_{\eta\eta}\in {\mathscr X},\notag\\
&\;\;\;D^{(j)}_{\eta}u_i(\cdot,-\ell_0/2)=D^{(j)}_{\eta}u_i(\cdot,\ell_0/2),~i=1,2,~j=0,1\Big\},
\label{dom-L-A}
\end{align}
is closable and its closure, denoted by $L_{\e}$, is sectorial and
$0$ is in the resolvent set of the restriction of $L_{\e}$ to
$(I-{\mathscr P})({\mathscr X})$ (see Theorem \ref{thm-generation}).
Hence, we can invert \eqref{final-w} using
$R(0,L_{\varepsilon})=(-L_{\varepsilon})^{-1}$, collecting linear
and nonlinear terms in $\psi$:
\begin{align}
{\bf z}=& R(0,L_{\varepsilon})
\left(-\psi_{\eta\eta}({\bf T}'-{\bf U})
+\varepsilon\psi_{\tau\eta\eta}\left ({\bf V}-{\bf T}-\frac{4}{3}{\bf U}\right)\right)\notag\\
& +R(0,L_{\varepsilon})\left(\frac{1}{2}(\psi_{\eta})^2({\bf T}'-{\bf
U}) +\varepsilon((\psi_{\eta})^2)_{\eta\eta} \left ({\bf V}-{\bf
T}-\frac{4}{3}{\bf U}\right )\right).
\label{eq-z}
\end{align}
\subsubsection{The fourth-order equation for the front}
\label{sub-3.5} Using \eqref{eq-z}, we can compute
$z_1(\cdot,0,\cdot)$ getting
\begin{align*}
z_1(\cdot,0,\cdot)=&-\left(R(0,L_{\varepsilon})
\left [\psi_{\eta\eta}({\bf T}'-{\bf U})\right ]\right)(\cdot,0,\cdot)\\
&+\varepsilon\left (R(0,L_{\varepsilon})\left [
\psi_{\tau\eta\eta}\left ({\bf V}-{\bf T}-\frac{4}{3}{\bf U}\right )\right ]\right )(\cdot,0,\cdot)\\
&+ \left\{R(0,L_{\varepsilon})
\left(\frac{1}{2}(\psi_{\eta})^2({\bf T}'-{\bf U}) +
\varepsilon((\psi_{\eta})^2)_{\eta\eta}\left ({\bf V}-{\bf
T}-\frac{4}{3}{\bf U}\right )\right)\right\}(\cdot,0,\cdot).
\end{align*}
Since $z_1$ is as smooth as $v_1$ is, we can differentiate the
previous formula twice with respect to $\eta$ obtaining
\begin{align}
D_{\eta\eta}z_1(\cdot,0,\cdot)=&-\left
(D_{\eta\eta}R(0,L_{\varepsilon})
\left [\psi_{\eta\eta}({\bf T}'-{\bf U})\right ]\right)(\cdot,0,\cdot)\notag\\
&+\varepsilon \left (D_{\eta\eta}R(0,L_{\varepsilon})\left [
\psi_{\tau\eta\eta}\left ({\bf V}-{\bf T}-\frac{4}{3}{\bf U}\right )\right ]\right )(\cdot,0,\cdot)\notag\\
&+ \frac{1}{2}\left\{D_{\eta\eta}R(0,L_{\varepsilon})
\left((\psi_{\eta})^2({\bf T}'-{\bf U})\right )\right\}(\cdot,0,\cdot)\notag\\
&+\varepsilon\left\{D_{\eta\eta}R(0,L_{\varepsilon})\left (
((\psi_{\eta})^2)_{\eta\eta}\left ({\bf V}-{\bf T}-\frac{4}{3}{\bf
U}\right )\right)\right\}(\cdot,0,\cdot).
\label{z1-0}
\end{align}

Estimate \eqref{estim-R} and our assumptions on $\psi$ (which
guarantee that the function $\psi_{\eta\eta}$ is continuously
differentiable in $[0,T]$ with values in
$C^{\alpha}([-\ell_0/2,\ell_0/2])$, see Remark \ref{rem-smooth})
show that the function $\left (D_{\eta\eta}R(0,L_{\varepsilon})\left
[ \psi_{\eta\eta}\left ({\bf V}-{\bf T}-\frac{4}{3}{\bf U}\right
)\right ]\right )(\cdot,0,\cdot)$ is continuously differentiable in
$[0,T]\times [-\ell_0/2,\ell_0/2]$ with respect to $\tau$ and its
derivative equals the function $\left
(D_{\eta\eta}R(0,L_{\varepsilon})\left [ \psi_{\tau\eta\eta}\left
({\bf V}-{\bf T}-\frac{4}{3}{\bf U}\right )\right ]\right
)(\cdot,0,\cdot)$. Hence, replacing \eqref{z1-0} into
\eqref{parabolic-4order} and taking the above remark into account,
we obtain that the function $\psi$ eventually solves the
fourth-order equation
\begin{align}
&\frac{\partial}{\partial \tau}{\mathscr B_{\e}}\psi
= {\mathscr S}_{\e}\psi +{\mathscr F}_{\e}((\psi_\eta)^2),\label{abstract}
\end{align}
where \setcounter{equation}{17}
\begin{align}
&{\mathscr B}_{\e}\psi=\psi-\frac{13\e}{3}\psi_{\eta\eta} +3\e^2
\left (D_{\eta\eta}R(0,L_{\varepsilon}) \left [\psi_{\eta\eta}\left
({\bf V}-{\bf T}-\frac{4}{3}{\bf U}\right )\right ]\right
)(\cdot,0,\cdot), \tag{\theequation a} \label{B-ve}
\\[2mm]
&{\mathscr S}_{\e}\psi=-3(1-\e)
\psi_{\eta\eta\eta\eta}-\psi_{\eta\eta}+3\e
\left(D_{\eta\eta}R(0,L_{\varepsilon})[ \psi_{\eta\eta}({\bf
T}'-{\bf U})]\right)(\cdot,0,\cdot), \tag{\theequation
b}\label{S-ve}
\\[2mm]
&{\mathscr F}_{\e}(\psi)= -3\e D_{\eta\eta}\left\{
R(0,L_{\varepsilon}) \left(\frac{1}{2}\psi({\bf T}'-{\bf U}) +
\e\psi_{\eta\eta}\left ({\bf V}-{\bf
T}-\frac{4}{3}{\bf U}\right )\right)\right\}(\cdot,0,\cdot)\notag\\
&\qquad\qquad+\frac{17\e}{6}\psi_{\eta\eta}-\frac{1}{2}\psi.
\tag{\theequation c}
\label{F-ve}
\end{align}
Clearly, \eqref{abstract} reduces to K--S when we set $\e=0$.

\subsection{Equivalence between Problem \eqref{eqn:v}-\eqref{cond-v-2} and Equation
\eqref{abstract}}

The following theorem states the equivalence of Problem
\eqref{eqn:v}-\eqref{cond-v-2} and Equation \eqref{abstract}.

\begin{theorem}
\label{thm-equivalence} Fix $\e,T>0$ and $\alpha\in (0,1/2)$.
Further, let $(v,\psi)\in {\mathscr V}_T\times {\mathscr
Y}_T$ be a solution to Problem
\eqref{eqn:v}-\eqref{cond-v-2} $($see Definition $\ref{def-YT})$.
Then, the function $\psi$ turns out to solve Equation
\eqref{abstract}.

Viceversa, if $\psi\in {\mathscr Y}_T$ is a solution to
Equation \eqref{abstract}, then there exists a function $v\in
{\mathscr V}_T$ such that the pair $(v,\psi)$ solves the Cauchy
problem \eqref{eqn:v}-\eqref{cond-v-2}.
\end{theorem}

\begin{proof}
In view of the arguments in Subsection \ref{sect-self-cons}, we just
need to show that to any solution $\psi\in {\mathscr Y}_T$ to
Equation \eqref{abstract} there corresponds a unique function $v\in
{\mathscr V}_T$ such that the pair $(v,\psi)$ solves Problem
\eqref{eqn:v}-\eqref{cond-v-2}. For this purpose, let ${\bf z}$ be
defined by \eqref{eq-z}. By assumptions, the functions
$\psi_{\eta\eta}$, $\psi_{\eta}^2$, $((\psi_{\eta})^2)_{\eta\eta}$
and $\psi_{\tau\eta\eta}$ are bounded in $[0,T]$ with values in the
space $C_{\sharp}^{\alpha}$. Moreover, the functions ${\bf T}'-{\bf
U}$ and ${\bf V}-{\bf T}-\frac{4}{3}{\bf U}$ are in $(I-{\mathscr
P})({\mathscr X})$. Hence, we can apply Theorem
\ref{thm-generation}(iv) and conclude that ${\bf z}(\tau,\cdot)$ is
in $D(L+\e A)$ (see \eqref{dom-L-A}) for any $\tau\in [0,T]$.

Clearly, the components $z_1$ and $z_2$ of ${\bf z}$ are continuous
in $[0,T]\times I_-$ and in $[0,T]\times I_+$, respectively. Let us
show that also the spatial derivatives (up to the second order) of
the functions $z_1$ and $z_2$ are continuous in $[0,T]\times I_-$
and $[0,T]\times I_+$. This follows from the estimate
\eqref{estim-R} provided one shows that the functions
$\psi_{\eta\eta}$, $((\psi_{\eta})^2)_{\eta\eta}$ and
$\psi_{\tau\eta\eta}$ are continuous in $[0,T]$ with values in
$C_{\sharp}^{\theta}$ for some $\theta\in (0,\alpha)$. Such a
property can be proved using an interpolation argument. Indeed, it
is well-known that, for any $\theta\in (0,\alpha)$, there exists a
positive constant $C$ such that
\begin{eqnarray*}
\|\psi\|_{C^{\theta}([-\ell_0/2,\ell_0/2])}\le
C\|\psi\|_{C([-\ell_0/2,\ell_0/2])}^{1-\theta/\alpha}\|\psi\|_{C^{\alpha}([-\ell_0/2,\ell_0/2])}^{\theta/\alpha},
\end{eqnarray*}
for any $\psi\in C^{\alpha}([-\ell_0/2,\ell_0/2])$ (see e.g.,
\cite{triebel}). Applying this estimate to the function
$\psi_{\eta\eta}(\tau_2,\cdot)-\psi_{\eta\eta}(\tau_1,\cdot)$, with
$\tau_1,\tau_2\in [0,T]$, shows that $\psi_{\eta\eta}$ is continuous
in $[0,T]$ with values in $C^{\theta}([-\ell_0/2,\ell_0/2])$ (and,
hence, in $C_{\sharp}^{\theta}$), for any $\theta\in (0,\alpha)$.
The same argument shows that the functions $\psi_{\eta}^2$,
$((\psi_{\eta})^2)_{\eta\eta}$ and $\psi_{\tau\eta\eta}$ are
continuous in $[0,T]$ with values in $C_{\sharp}^{\theta}$ as well.
Finally, since ${\bf z}(\tau,\cdot)$ belongs to $D(L+\e A)$ for any
$\tau\in [0,T]$, the functions $(\tau,x,\eta)\mapsto
e^{-\frac{x}{2}}D^{(i)}_xz_1(\tau,x,\eta)$ and $(\tau,x,\eta)\mapsto
e^{-\frac{x}{2}}D^{(i)}_{\eta}z_2(\tau,x,\eta)$ are bounded in
$[0,T]\times I_-$ and in $[0,T]\times I_+$, respectively, for any
$i=0,1,2$.

The function ${\bf z}$ will represent the component along
$(I-{\mathscr P})({\mathscr X})$ of the function ${\bf
v}-\varepsilon{\mathscr N}(\psi_{\tau}+(\psi_y)^2)$, where
$v_1(\cdot,x,\cdot)=v(\cdot,x,\cdot)\chi_{(-\infty,0]}(x)$,
$v_2(\cdot,x,\cdot)=v(\cdot,x,\cdot)\chi_{[0,+\infty)}(x)$ and $v$
is the solution to Problem \eqref{eqn:v}--\eqref{cond-v-2} we are
looking for. The computations in Subsection \ref{sect-self-cons}
suggest to set ${\bf v}:={\bf w}+\varepsilon{\mathscr
N}(\psi_{\tau}+(\psi_y)^2):=a{\bf U}+\varepsilon{\bf
z}+\varepsilon{\mathscr N}(\psi_{\tau}+(\psi_y)^2)$, where
\begin{equation}
a=-3\varepsilon
z_1(\cdot,0,\cdot)+3(\varepsilon-1)\psi_{\eta\eta}
+3\varepsilon\psi_{\tau}+\frac{3}{2}\varepsilon(\psi_{\eta})^2.
\label{def-a}
\end{equation}
Using Formulae \eqref{Q-1}-\eqref{Q-3} and \eqref{z1-0} we can show
that
\begin{eqnarray*}
{\mathscr P}(\e {\bf F}_0 -\varepsilon {\bf w}_{\eta\eta}
-\varepsilon {\mathscr L}{\mathscr N}(\psi_{\tau}+(\psi_{\eta})^2)
-\varepsilon^2{\mathscr
N}((\psi_{\tau}+(\psi_{\eta})^2)_{\eta\eta}))=0.
\end{eqnarray*}
Hence, the function ${\bf v}$ solves the equation
\begin{align*}
{\mathscr L}{\bf v}=&
L(a{\bf U})+\e L{\bf z}+\varepsilon{\mathscr L}{\mathscr N}(\psi_{\tau}+(\psi_y)^2)\\
=&(I-{\mathscr P})\{\e {\bf F}_0 -\varepsilon {\bf w}_{\eta\eta}
-\varepsilon {\mathscr L}{\mathscr N}(\psi_{\tau}+(\psi_{\eta})^2)-\varepsilon^2{\mathscr N}((\psi_{\tau}+(\psi_{\eta})^2)_{\eta\eta})\}\\
&+\e{\mathscr L}({\mathscr N}(\psi_{\tau}+(\psi_y)^2))\\
=&\e {\bf F}_0 -\varepsilon {\bf w}_{\eta\eta}
-\varepsilon^2{\mathscr N}((\psi_{\tau}+(\psi_{\eta})^2)_{\eta\eta})\\
=&\e {\bf F}_0-\e {\bf v}_{\eta\eta}.
\end{align*}
Moreover, it is easy to check that ${\bf v}$ satisfies also the
boundary conditions of the Cauchy problem \eqref{pb-fin-u}.

Clearly, the function $v$ defined above belongs to ${\mathscr V}_T$
and the pair $(v,\psi)$ solves the differential equation
\eqref{eqn:v}. Using the second boundary condition in
\eqref{pb-fin-u}, it follows immediately that $(v,\psi)$ satisfies
condition \eqref{cond-v-1}. Finally, to check condition
\eqref{cond-v-2} it suffices to use \eqref{def-a}, recalling that
${\mathscr N}(\psi_{\tau}+(\psi_{\eta})^3)$ vanishes when $\eta=0$.
This completes the proof.
\end{proof}

\subsection{The equation for the remainder}
\label{sect:4-3} In view of Theorem \ref{thm-equivalence}, in the
rest of the paper we deal only with Equation \eqref{abstract} with
periodic boundary conditions.
 To begin with, we recall the following
result about K--S:
\begin{theorem}
\label{thm-4.1} Let $\Phi_0\in C_{\sharp}^{6+\alpha}$ for some
$\alpha\in (0,1/2)$. Then, the Cauchy problem
\begin{equation}
\left\{
\begin{array}{lll}
\Phi_\tau(\tau,\eta)= -3\Phi_{\eta\eta\eta\eta}(\tau,\eta)-
\Phi_{\eta\eta}(\tau,\eta)
- \frac{1}{2}(\Phi_\eta(\tau,\eta))^2, &\tau\ge 0, & |\eta|\le\frac{\ell_0}{2},\\[2mm]
D_{\eta}^k\Phi(\tau,-\ell_0/2)=D_{\eta}^k\Phi(\tau,\ell_0/2), &\tau\ge 0, & k=0,1,2,3,\\[2mm]
\Phi(0,\eta)=\Phi_0(\eta), &&|\eta|\le \frac{\ell_0}{2},
\end{array}
\right. \label{4order}
\end{equation}
admits a unique solution $\Phi\in C^{1,4}([0,+\infty)\times
[-\ell_0/2,\ell_0/2])$. In fact, $\Phi\in {\mathscr Y}_T$
for any $T>0$.
\end{theorem}

Most of the literature is about the differentiated version of K--S. For this reason and the
reader's convenience, we provide a full proof of Theorem \ref{thm-4.1} in the appendix.

According to the {\it Ansatz}, we split
\begin{eqnarray*}
\psi= \Phi+ \e\rho_\e,
\end{eqnarray*}
which defines the remainder $\rho_{\e}$. To avoid cumbersome
notation, we simply write $\rho$ for $\rho_\e$. From Theorem
\ref{thm-4.1} we know that $\rho\in {\mathscr Y}_T$ (see Definition
\ref{def-YT}) and it solves the equation
\begin{equation}
\label{eqn:rho} \frac{\partial}{\partial \tau}{\mathscr
B}_{\e}(\rho) = {\mathscr S}_{\e}(\rho) -\Phi_\eta\rho_\eta
-\frac{\e}{2}(\rho_\eta)^2 + {\mathscr
G}_{\e}((\Phi_{\eta}+\e\rho_{\eta})^2) +{\mathscr H}_{\e}(\Phi),
\end{equation}
where
\begin{align*}
{\mathscr G}_{\e}(\xi) =&
\frac{17}{6}\xi_{\eta\eta}-3\left\{D_{\eta\eta} R(0,L_{\varepsilon})
\left(\frac{1}{2}\xi({\bf T}'-{\bf U}) + \e\xi_{\eta\eta}\left ({\bf
V}-{\bf T}-\frac{4}{3}{\bf U}\right )\right)\right\}(\cdot,0,\cdot);\\[1mm]
{\mathscr H}_{\e}(\Phi)=& 3\Phi_{\eta\eta\eta\eta}
+3\left(D_{\eta\eta}R(0,L_{\varepsilon})[ \Phi_{\eta\eta}({\bf
T}'-{\bf U})]\right)(\cdot,0,\cdot)
+\frac{13}{3}\Phi_{\tau\eta\eta}\\
&-3\e\left (D_{\eta\eta}R(0,L_{\varepsilon}) \left [\Phi_{\tau\eta\eta}\left
({\bf V}-{\bf T}-\frac{4}{3}{\bf U}\right )\right ]\right
)(\cdot,0,\cdot).
\end{align*}

Equation \eqref{eqn:rho} on $[-\ell_0/2, \ell_0/2]$ is supplemented
by periodic boundary conditions and by an initial condition $\rho_0$
at $\tau=0$. For simplicity, to avoid lengthly computations, we take
hereafter $\rho_0=0$, namely, $\psi(0,\cdot)=\Phi(0,\cdot)=\Phi_0$.
In other words, the front $\psi$ and the solution of K--S start from
the same configuration, which is physically reasonable. More general
{\it compatible} initial data can be considered as in
\cite{BFHS,BFHLS}.

\section{Local in time solvability of Equation \eqref{eqn:rho}}
\setcounter{equation}{0} \label{sect:AP} As it has been remarked in
the introduction, except for small $\ell_0$, where the TW is stable,
global existence of $\rho$ is not granted.

In this section, we prove the following local in time existence and uniqueness result.

\begin{theorem}
\label{thm-exist-eps} For any $\e\in (0,1/2]$, there exist
$T_{\e}>0$ and a unique solution $\rho$ to Equation \eqref{eqn:rho}
which belongs to ${\mathscr Y}_{T_{\e}}$ $($see Definition
$\ref{def-YT})$ and vanishes at $\tau=0$.
\end{theorem}

The proof is rather long and needs many preliminary results.
For this reason, we split it in several
steps. Before entering the details, we sketch here the strategy of
the proof.

As a first step, for any fixed $\e>0$, we transform Equation
\eqref{eqn:rho} into a semilinear equation associated with a
sectorial operator. Employing classical tools from the theory of
analytic semigroups we prove that such a semilinear equation admits
a unique solution $\rho=\rho_{\e}$ defined in some time domain
$[0,T_{\e}]$, which vanishes at $\tau=0$. Using some bootstrap
arguments, we then regularize $\rho$, showing that it actually
belongs to ${\mathscr Y}_{T_{\e}}$. These regularity properties of
$\rho$ allow us to show that it is in fact a solution to Equation
\eqref{eqn:rho}.

\subsection{The semilinear equation}

In this subsection, we show that we can transform Equation
\eqref{eqn:rho} into a semilinear equation associated with a second
order elliptic operator. We obtain it inverting the operator
${\mathscr B}_{\e}$ in \eqref{B-ve}, i.e., the operator defined by
\begin{align*}
{\mathscr B}_{\e}\psi=\psi-\frac{13\e}{3}\psi_{\eta\eta} +3\e^2
\left (D_{\eta\eta}R(0,L_{\varepsilon}) \left [\psi_{\eta\eta}\left
({\bf V}-{\bf T}-\frac{4}{3}{\bf U}\right )\right ]\right
)(\cdot,0,\cdot).
\end{align*}
By Theorem \ref{thm-generation} and the results in the proof of
Theorem \ref{thm-equivalence}, we know that the operator ${\mathscr
B}_{\e}$ is well-defined in $C_{\sharp}^{2+\theta}$ for any
$\theta\in (0,1)$. We will show that ${\mathscr B}_{\e}$ can be
extended to the whole of $C_{\sharp}^2$ with an operator which is
invertible. For this purpose, we compute the symbol of the operator
${\mathscr B}_{\e}$.

Throughout the section, given a function $f:J\times [-\ell_0/2,\ell_0/2]\to\R$, where
$J\subset\R$ is an interval, we denote by $\hat f(x,k)$ the $k$-th Fourier coefficient
of the function $f(x,\cdot)$. Moreover, we set
\begin{equation}
X_{\e,k}=\sqrt{1+4\e\lambda_k},\qquad\;\,k\in\N\cup\{0\}.
\label{Xk}
\end{equation}
\begin{lemma}
\label{operatorB-1} Fix $\e\in (0,1/2]$. Then, the $k$-th Fourier
multiplier $b_{\e,k}$ of the operator ${\mathscr B}_{\e}$ is given
by
\begin{equation}
b_{\e,k}=\frac{3}{4}\frac{(X_{\e,k}+1)(X_{\e,k}^2+2X_{\e,k}-1)}{X_{\e,k}+2}
\sim 3\e\lambda_k\qquad (k\to +\infty). \label{bek}
\end{equation}
\end{lemma}

\begin{proof}
Even if the proof can be obtained arguing as in the proof of
\cite[Prop. 4.2]{BHL08}, for the reader's convenience we go into
details.

The main step of the proof is the computation of the symbols of the
two operators $\varphi\mapsto {\bf u}:=\left (R(0,L_{\varepsilon})
\left [\varphi\left ({\bf V}-{\bf T}-\frac{4}{3}{\bf U}\right
)\right ]\right )(0,\cdot)$ and $\varphi\mapsto {\bf
v}:=\left(R(0,L_{\varepsilon})[ \varphi({\bf T}'-{\bf
U})]\right)(0,\cdot)$, for any $\e>0$. To enlighten a bit the
notation, throughout the proof we do not stress explicitly the
dependence on the quantities we consider on $\e$.

We claim that
\begin{align}
&\hat u_1(0,k)=
-\frac{4}{9}\frac{4X_k+7}{(X_k+1)^2(X_k+2)}
\hat\varphi(k),\qquad k=0,1,\ldots,\label{u1k}
\\[2mm]
&\hat v_1(0,k)= \frac{2}{3}\frac{1}{(X_k+1)(X_k+2)}
\hat\varphi(k),\qquad\;\,k=0,1,\ldots \label{u2k}
\end{align}
We limit ourselves to dealing with the function ${\bf u}$, since the
same arguments apply to the function ${\bf v}$. Let us first assume
that $\varphi$ is smooth enough. Since the function ${\bf V}-{\bf
T}-\frac{4}{3}{\bf U}$ belongs to $(I-{\mathscr P})({\mathscr X})$,
from Proposition \ref{thm-generation}(iii) it follows that ${\bf
u}\in D(L+\e A)$, so that ${L}{\bf u}+\e{A}{\bf u}=-({\bf V}-{\bf
T}-\frac{4}{3}{\bf U})\varphi$. Moreover, the function $\hat{\bf
u}(\cdot,k)$ belongs to $(I-{\mathscr P})(D(L))$ and solves the
equation $(\e\lambda_k-L)\hat {\bf u}(\cdot,k)=({\bf V}-{\bf
T}-\frac{4}{3}{\bf U})\hat\varphi(k)$ for any $k=0,1,\ldots$ Since
$\lambda_k$ is in the resolvent set of the operator ${L}$ for any
$k=0,1,\ldots$, by Theorem \ref{thm:1} it follows that
\begin{equation}
\hat{\bf u}(\cdot,k)=R(\e\lambda_k,{L})\left ({\bf V}-{\bf
T}-\frac{4}{3}{\bf U}\right )\hat\varphi(k),\qquad\;\, k=0,1,\ldots
\label{form-u}
\end{equation}
Formula \eqref{form-u} can be extended to any function $\varphi\in
C_{\sharp}$ by a straightforward approximation argument.

From formula \eqref{3.1-b} it is immediate to check that
\begin{align*}
(R(\e\lambda_k,L){\bf f})_1(0,\cdot)=&
\frac{2\e\lambda_k}{1+(2\e\lambda_k-1)X_k}\notag\\
&\qquad\times \left [
\int_{-\infty}^0e^{-\nu_{1,k}t}f_1(t,\cdot)dt
+\int_0^{+\infty}e^{-\nu_{2,k}t}f_2(t,\cdot)dt\right ],
\end{align*}
for any ${\bf f}=(f_1,f_2)\in {\mathscr X}$, where
\begin{eqnarray*}
\nu_{1,k}=\frac{1}{2}-\frac{1}{2}X_k,\qquad
\nu_{2,k}=\frac{1}{2}+\frac{1}{2}X_k,\qquad\;\,k=0,1,\ldots
\end{eqnarray*}
Hence, from the very definition of the functions ${\bf V}$, ${\bf
T}$ and ${\bf U}$ (see \eqref{funct-u2-a} and \eqref{funct-u2-b}),
we get
\begin{align*}
\left \{( \e\lambda_k+D_x-D^2_x)^{-1}\left ({\bf V}-{\bf
T}-\frac{4}{3}{\bf U}\right )\right \}_1(0)=&
-\frac{8\nu_{2,k}^2-5\nu_{2,k}-3}{9\nu_{2,k}^3X_k}\\[1mm]
=& -\frac{4}{9}
\frac{(4X_k+7)(X_k-1)}{X_k(X_k+1)^3}.
\end{align*}
Since $0$ is in the resolvent set
of the restriction of $L$ to $(I-{\mathscr P})({\mathscr X})$, we
can extend the previous formula, by continuity, to $\lambda=0$.
Thus,
\begin{align*}
\hat u_1(0,k)&=-\frac{4}{9}\frac{2\e\lambda_k}
{1+(2\e\lambda_k-1)X_k}
\frac{(4X_k+7)(X_k-1)}{(X_k+1)^3}
\hat\varphi(k)\\
&=\frac{4}{9}\frac{(4X_k+7)(X_k-1)}{(X_k+2)(X_k+1)^2}\hat\varphi(k),
\end{align*}
for any $k=0,1,\ldots$, and the assertion follows.

Now, using Formulae \eqref{u1k} and \eqref{u2k}, it is immediate to
complete the proof.
\end{proof}
\begin{prop}
\label{prop-inv-B} For any $\e\in (0,1/2]$, the operator ${\mathscr
B}_{\e}$ is invertible from $C_{\sharp}^{2+\theta}$ into
$C_{\sharp}^{\theta}$ for any $\theta\in (0,1)$.
\end{prop}

\begin{proof}
From Lemma \ref{operatorB-1}, we know that $b_{\e,k}\neq 0$, for any
$k\in\N\cup\{0\}$. Hence, operator ${\mathscr B}_{\e}$ admits a
realization in $L^2$ which is invertible from $H^2$ into $L^2$. We
still denote by ${\mathscr B}_{\e}$ such a realization. To prove
that ${\mathscr B}_{\e}$ is invertible from $C_{\sharp}^2$ into
$C_{\sharp}$, let us fix $f\in C_{\sharp}$ and let $u\in H^2$ be the
unique solution to the equation ${\mathscr B}_{\e}u=f$. Taking
\eqref{bek} into account, it is immediate to check that, we can
split ${\mathscr B}_{\e} =-3\e D_{\eta\eta}+\overline {\mathscr
B}_{\e}$, where $\overline {\mathscr B}_{\e}$ is a bounded operator,
whose symbol $(\overline b_{\e,k})$ satisfies
\begin{eqnarray*}
\overline b_{\e,k}\sim \frac{3}{2}\sqrt{\e\lambda_k},\qquad (k\to
+\infty).
\end{eqnarray*}
It follows that the function $\overline {\mathscr B}_{\e}(u)$ is in
$C_{\sharp}$. By difference $u_{\eta\eta}$ is in $C_{\sharp}$ as
well. A bootstrap argument can now be used to prove that, if $f\in
C_{\sharp}^{\theta}$, then $u\in C_{\sharp}^{2+\theta}$.
\end{proof}

In view of Proposition \ref{prop-inv-B}, we can invert the operator
${\mathscr B}_{\e}$ from $C_{\sharp}^{2+\theta}$ into
$C_{\sharp}^{\theta}$ for any $\theta\in (0,1)$, getting the
following equation for $\rho$:
\begin{equation}
\rho_{\tau}(\tau,\cdot) = {\mathscr R}_{\e}(\rho(\t,\cdot)) +
{\mathscr K}_{\e}(\tau,\rho_{\eta}(\tau,\cdot)), \qquad\;\,\tau\in
[0,T], \label{eq-2nd-order}
\end{equation}
where
\begin{align}
{\mathscr R}_{\e}(\rho)=&{\mathscr B}_{\e}^{-1}({\mathscr
S}_{\e}(\rho)),\notag\\[2mm]
{\mathscr K}_{\e}(\tau,\rho)=&{\mathscr B}_{\e}^{-1}({\mathscr
G}_{\e}((\Phi_{\eta}(\tau,\cdot))^2))-{\mathscr
B}_{\e}^{-1}(\Phi_\eta(\tau,\cdot)\rho) +2\e {\mathscr
B}_{\e}^{-1}({\mathscr
G}_{\e}(\Phi_{\eta}(\tau,\cdot)\rho))\notag\\
&-\frac{\e}{2}{\mathscr B}_{\e}^{-1}(\rho^2) + \e^2{\mathscr
B}_{\e}^{-1}({\mathscr G}_{\e}(\rho^2))+3{\mathscr B}_{\e}^{-1}(\Phi_{\eta\eta\eta\eta})\notag\\
&+3{\mathscr B}_{\e}^{-1}\left
(\left(D_{\eta\eta}R(0,L_{\varepsilon})[ \Phi_{\eta\eta}({\bf
T}'-{\bf U})]\right)(\cdot,0,\cdot)\right )
+\frac{13}{3}{\mathscr B}_{\e}^{-1}\left (\Phi_{\tau\eta\eta}\right )\notag\\
&-3\e{\mathscr B}_{\e}^{-1}\left (\left
(D_{\eta\eta}R(0,L_{\varepsilon}) \left [\Phi_{\tau\eta\eta}\left
({\bf V}-{\bf T}-\frac{4}{3}{\bf U}\right )\right ]\right
)(\cdot,0,\cdot)\right ). \label{K-eps}
\end{align}

\subsection{Solving Equation \eqref{eq-2nd-order}}
\label{subsect-symbol} Here, we prove an existence and uniqueness
result for  Equation \eqref{eq-2nd-order} with initial condition
$\rho(0,\cdot)=0$. For this purpose, we need to thoroughly study the
operators ${\mathscr R}_{\e}$ and ${\mathscr K}_{\e}$. To enlighten
the notation, we do not stress explicitly the dependence on $\e$ of
the symbols of the operators we are going to consider. In
particular, we simply write $X_k$ for $X_{\e,k}$ (see \eqref{Xk}).

We begin by considering the operator ${\mathscr R}_{\e}$. Taking
Proposition \ref{prop-inv-B} and Theorem \ref{thm-generation}(iii)
into account, it is immediate to check that the operator ${\mathscr
R}_{\e}$ is well-defined in $C_{\sharp}^4$. Actually, we show that
it can be extended to $C_{\sharp}^1\cap C^2$ with a bounded operator
which is sectorial.

\begin{prop}
\label{prop-5.4} For any $\e\in (0,1/2]$, the operator ${\mathscr
R}_{\e}$ can be extended with a sectorial operator $R_{\e}$ having
$C_{\sharp}^1\cap C^2$ as domain. Moreover,
$D_{R_{\e}}(\theta,\infty)=C_{\sharp}^{2\theta}$ for any $\theta\in
(0,1)\setminus\{1/2\}$, with equivalence of the corresponding norms.
\end{prop}

\begin{proof}
To begin with, we compute the symbol of the operator ${\mathscr
S}_{\e}$. We have:
\begin{equation}
s_k=\frac{3(X_k-1)(X_k+1)^2\{(\e-1)X_k^2+(\e-1)X_k+2\}}
{16\e^2(X_k+2)} \sim 48 (\e-1)\e^2\lambda_k^2,
\label{sek}
\end{equation}
as $k\to +\infty$. Hence, from \eqref{bek} and \eqref{sek} it
follows that the $k$-th symbol of the operator ${\mathscr R}_{\e}$
is
\begin{eqnarray*}
r_k=\frac{(X_k^2-1)\{(\e-1)X_k^2
+(\e-1)X_k+2\}}{4\e^2(X_k^2+2X_k-1)},\qquad\;\,k=0,1,\ldots
\end{eqnarray*}
At any fixed $\e\in (0,1/2]$ $r_k\sim (1-\e^{-1})\lambda_k$ as $k\to
+\infty$. Hence, we can split
\begin{eqnarray*}
{\mathscr R}_{\e}\phi=\frac{1-\e}{\e}\varphi_{yy}+{\mathscr
R}_{\e}^{(1)}\phi,
\end{eqnarray*}
where the symbol of ${\mathscr R}_{\e}^{(1)}$ is
\begin{eqnarray*}
r_k^{(1)}=\frac{(X_k^2-1)\{(1-\e)X_k+\e+1\}}{4\e^2(X_k^2+2X_k-1)}\sim
\frac{(1-\e)\sqrt{\e}}{2\e^2}\sqrt{\lambda_k},\qquad\;\,(k\to
+\infty).
\end{eqnarray*}
We claim that the operator ${\mathscr R}_{\e}^{(1)}$ admits a
realization in $C([-\ell_0/2,\ell_0/2])$ which is a bounded operator
mapping $C_{\sharp}^{1+\alpha}$ (for any $\alpha\in (0,1)$) into
$C_{\sharp}$. As a first step, we observe that, due to the
characterization of the spaces $H^s$ given in \eqref{space-H}, the
operator ${\mathscr R}_{\e}^{(1)}$ admits a realization $
R_{\e}^{(1)}$ which is bounded from $H^s$ into $H^{s-1}$ for any
$s\ge 1$. It is well-known that $C_{\sharp}^m\subset H^m\subset
C_{\sharp}^{m-1/2}$ with continuous embeddings, for any $m>1/2$ such
that $m-1/2\notin\N$. As a consequence, the operator $R_{\e}^{(1)}$
is bounded from $C_{\sharp}^s$ into $C_{\sharp}^{s-3/2}$ for any
$s>3/2$ such that $s-3/2\notin\N$. Therefore, the operator
${\mathscr R}_{\e}$ can be extended with a bounded operator $R_{\e}$
from $D(R_{\e})=C_{\sharp}^1\cap C^2$ into
$C([-\ell_0/2,\ell_0/2])$.

Let us now prove that $R_{\e}$ is sectorial. For this purpose, we
note that $C_{\sharp}^{\theta}$ belongs to the class $J_{\theta/2}$
between $C([-\ell_0/2,\ell_0/2])$ and $C_{\sharp}^1\cap C^2$, for
any $\theta\in (0,2)$, i.e., there exists a positive constant $K$
such that
\begin{equation}
\|f\|_{C^{\theta}([-\ell_0/2,\ell_0/2])}\le
K\|f\|_{C([-\ell_0/2,\ell_0/2])}^{\frac{2-\theta}{2}}
\|f\|_{C^2}^{\frac{\theta}{2}}, \label{stim-interp-theta}
\end{equation}
for any $f\in C_{\sharp}^1\cap C^2$, and the realization of the
second order derivative in $C([-\ell_0/2,\ell_0/2])$ with domain
$C_{\sharp}^1\cap C^2$ is sectorial. Hence, we can apply \cite[Prop.
2.4.1(i)]{lunardi} and conclude that the operator $R_{\e}$ is
sectorial in $C([-\ell_0/2,\ell_0/2])$. From the above arguments, it
is now clear that the graph norm of $R_{\e}$ is equivalent to the
Euclidean norm of $C_{\sharp}^1\cap C^2$. Hence, \cite[Prop.
2.2.2]{lunardi} implies that
$D_{R_{\e}}(\theta,\infty)=C_{\sharp}^{2\theta}$ for any $\theta\in
(0,1)\setminus\{1/2\}$.
\end{proof}

We now consider the operator ${\mathscr K}_{\e}$. From Proposition
\ref{prop-inv-B} and Theorems \ref{thm-4.1},
\ref{thm-generation}(iii), we know that the operator ${\mathscr
K}_{\e}$ is continuous from $C_{\sharp}^{2+\alpha}$ into
$[0,+\infty)\times C_{\sharp}$ for any $\alpha>0$. Let us show that
it can be extended to a larger domain.

\begin{prop}
\label{prop-H} For any $\e\in (0,1/2]$, the operator ${\mathscr
K}_{\e}$ can be extended with a continuous operator mapping
$C_{\sharp}^s$ into $[0,+\infty)\times C_{\sharp}^s$ for any $s\in
[0,2]$. Moreover, for any $T>0$ and any $r>0$, there exists a
positive constant $K=K(T,r)$ such that
\begin{equation*}
\|{\mathscr K}_{\e}(\tau_2,\psi)-{\mathscr
K}_{\e}(\tau_1,\psi)\|_{\infty}+\|{\mathscr
K}_{\e}(\tau,\psi)-{\mathscr K}_{\e}(\tau,\xi)\|_{\infty} \le K\left
(|\tau_2-\tau_1|+ \|\psi-\xi\|_{\infty}\right ),
\end{equation*}
for any $\tau,\tau_1,\tau_2\in [0,T]$ and any $\psi,\xi\in
B(0,r)\subset C_{\sharp}$.
\end{prop}

\begin{proof}
As a first step, we observe that, using Formulae \eqref{u1k} and
\eqref{u2k}, one can easily show that the $k$-th symbol $g_k$ of the
operator ${\mathscr G}_{\e}$ is
\begin{equation}
g_k=-\lambda_k\frac{3X_k^2+15X_k+4}{2(X_k+1)(X_k+2)}.
\label{fek}
\end{equation}
From \eqref{bek} and \eqref{fek}, it follows that the symbol of the operator
${\mathscr Z}_{\e}:={\mathscr B}_{\e}^{-1}{\mathscr G}_{\e}$ is
\begin{eqnarray*}
z_k=-\frac{2}{3}\lambda_k\frac{3X_k^2+15X_k+4}{(X_k+1)^2(X_k^2+2X_k-1)}=
-\frac{1}{2\e}+z^{(1)}_k,
\end{eqnarray*}
where
\begin{equation}
z^{(1)}_k\sim -\frac{1}{4\sqrt{\e^3\lambda_k}},\qquad k\to +\infty.
\label{symb-z1}
\end{equation}
Hence, we can write
\begin{eqnarray*}
{\mathscr Z}_{\e}=-\frac{1}{2\e}Id+{\mathscr Z}_{\e}^{(1)}.
\end{eqnarray*}
Formula \eqref{symb-z1} shows that the operator ${\mathscr
Z}_{\e}^{(1)}$ is bounded from $H^s$ into $H^{s+1}$ for any $s\ge
0$. Hence, it is bounded from $C_{\sharp}^s$ into
$C_{\sharp}^{s+\theta}$ for any $s\in\N\cup\{0\}$ and any $\theta\in
(0,1/2)$. As a byproduct, the operator ${\mathscr Z}_{\e}$ is
bounded from $C_{\sharp}^s$ into itself for any $s\ge 0$. Since,
\begin{align*}
{\mathscr K}_{\e}(\tau,\psi)=&{\mathscr
Z}_{\e}((\Phi_{\eta}(\tau,\cdot))^2)- {\mathscr
B}_{\e}^{-1}(\Phi_\eta(\tau,\cdot)\psi) +2\e
{\mathscr Z}_{\e}(\Phi_{\eta}(\tau,\cdot)\psi)-\frac{\e}{2}{\mathscr B}_{\e}^{-1}(\psi^2)\notag\\
& + \e^2 {\mathscr Z}_{\e}(\psi^2)
+3{\mathscr B}_{\e}^{-1}(\Phi_{\eta\eta\eta\eta})+\frac{13}{3}{\mathscr B}_{\e}^{-1}(\Phi_{\tau\eta\eta})\notag\\
&+3{\mathscr B}_{\e}^{-1}\left\{\left (D_{\eta\eta}R(0,L_{\e})[\Phi_{\eta\eta}({\bf T}'-{\bf U})]\right )(\cdot,0,\cdot)
\right\}\notag\\
&-3\e{\mathscr B}_{\e}^{-1}\left\{\left (D_{\eta\eta}R(0,L_{\e})\left [\Phi_{\tau\eta\eta}
\left ({\bf V}-{\bf T}-\frac{4}{3}{\bf U}
\right )\right ]\right )(\cdot,0,\cdot)\right\},
\end{align*}
for any $\psi\in C_{\sharp}$, taking Proposition \ref{prop-inv-B}
and Theorem \ref{thm-4.1} into account, the assertion follows at
once.
\end{proof}

From all the previous results, we get the following:
\begin{theorem}
\label{thm-semilinear} For any $\e\in (0,1/2]$, Equation
\eqref{eq-2nd-order} admits a unique solution $\rho$ defined in a
maximal time domain $[0,T_{\e})$ which vanishes at $\tau=0$, belongs
to $C^{1,2}([0,T_{\e})\times [-\ell_0/2,\ell_0/2])$ and satisfies
$D_{\eta}^{(j)}\rho(\cdot,\ell_0/2)\equiv
D_{\eta}^{(j)}\rho(\cdot,\ell_0/2)$ for $j=0,1$.
\end{theorem}

\begin{proof}
Combining \cite[Theorems 7.1.2 and 4.3.8]{lunardi}, we can easily
show that Equation \eqref{eq-2nd-order} admits a unique solution
$\rho$, defined in a maximal time domain $[0,T_{\e})$, which belongs
to $C^{1,\beta}([0,T_{\e})\times [-\ell_0/2,\ell_0/2])$ for any
$\beta<2$, vanishes at $\tau=0$, and satisfies
$D_{\eta}^{(j)}\rho(\cdot,\ell_0/2)\equiv
D_{\eta}^{(j)}\rho(\cdot,\ell_0/2)$ for $j=0,1$. Moreover,
$R_{\e}(\rho)$ is continuous in $[0,T_{\e})\times
[-\ell_0/2,\ell_0/2]$. Since $R_{\e}(\rho)$ and
$\frac{1-\e}{\e}D_{\eta\eta}$ differ in the lower order operator
$R_{\e}^{(1)}$ (see the proof of Proposition \ref{prop-5.4}),
$\rho_{\eta\eta}\in C([0,T_{\e})\times [-\ell_0/2,\ell_0/2])$, as
well, and this completes the proof.
\end{proof}

\subsection{Proof of Theorem \ref{thm-exist-eps}}
\label{subsect-proof}
In this subsection, using some bootstrap
arguments, we show that the solution $\rho$ to the Equation
\eqref{eq-2nd-order}, whose existence has been guaranteed in Theorem
\ref{thm-semilinear} is actually a solution to Equation
\eqref{eqn:rho}. Of course, we just need to show that both the
functions $\rho_{\eta\eta}$ and $\rho_{\tau}$ belong to
$C^{0,2}([0,T_{\e})\times [-\ell_0/2,\ell_0/2])$. Throughout the
proof, we assume that $T'$ is any arbitrarily fixed real number in
the interval $[0,T_{\e})$.

To begin with, we observe that, from \eqref{K-eps}, Theorem
\ref{thm-4.1} and Proposition \ref{prop-inv-B}, it follows
immediately that the function ${\mathscr K}_{\e}$ is continuous in
$[0,T']$ with values in $D_{R_{\e}}(\theta,\infty)$ for any
$\theta\in (0,1)$ (see Proposition \ref{prop-5.4}), since the
operator ${\mathscr Z}_{\e}$ is bounded from $C_{\sharp}^s$ into
itself for any $s\ge 0$, by the proof of Proposition \ref{prop-H}
and an interpolation argument. Therefore, we can apply \cite[Thm.
4.3.8]{lunardi} and conclude that $\rho_{\tau},R_{\e}(\rho)$ are
bounded in $[0,T']$ with values in $C_{\sharp}^{2\theta}$ for any
$\theta\in (0,1)$. As it has been already remarked, $R_{\e}(\rho)$
and $\frac{1-\e}{\e}D_{\eta\eta}$ differ in the lower order operator
$R_{\e}^{(1)}$. Hence, $\rho_{\eta\eta}\in
C^{0,2\theta}([0,T']\times [-\ell_0/2,\ell_0/2])$ for any $\theta$
as above. In particular, $\rho_{\tau}$ and $\rho_{\eta\eta}$ are
continuously differentiable with respect to $\eta$ in $[0,T']\times
[-\ell_0/2,\ell_0/2]$.

Let us now set $\zeta=\rho_{\eta}$. The previous results show that
$\zeta\in C([0,T_{\e});C_{\sharp}^2)\cap C^1([0,T_{\e});C_{\sharp})$
and $\zeta_{\tau}=\rho_{\eta\tau}$. Clearly, $\zeta_{\tau}=
(1-\e^{-1})\zeta_{\eta\eta}+D_{\eta}R_{\e}^{(1)}\rho+
D_{\eta}{\mathscr K}_{\e}(\cdot,\rho_{\eta})$ in $[0,T_{\e})\times
[-\ell_0/2,\ell_0/2]$, and $\zeta(0,\cdot)\equiv 0$. Since the
operator $R_{\e}^{(1)}$ is bounded from $C_{\sharp}^{3+\theta}$ into
$C_{\sharp}^{3/2+\theta}$ for any $\theta\in (0,1)\setminus\{1/2\}$,
the function $D_{\eta}R_{\e}^{(1)}(\rho)$ is bounded in $[0,T']$
with values in $C_{\sharp}^{\alpha}$ for any $\alpha<3/2$. Since the
function $D_{\eta}{\mathscr K}_{\e}(\cdot,\rho_{\eta}^2)$ is bounded
in $[0,T']$ with values in $C_{\sharp}^{\alpha}$ as well, it follows
that $D_{\eta}R_{\e}^{(1)}(\rho)+ D_{\eta}{\mathscr
K}_{\e}(\cdot,\rho_{\eta}(\tau,\cdot))$ is bounded in $[0,T']$ with
values in $C_{\sharp}^{\alpha}$ for any $\alpha$ as above. Hence,
Theorem 4.3.9(iii) of \cite{lunardi} implies that the functions
$\zeta_{\tau}$ and $\zeta_{\eta\eta}$ are bounded (in fact,
continuous) in $[0,T']$ with values in $C_{\sharp}^{\alpha}$. This
completes the proof.

\section{Uniform existence of $\rho$ and proof of the main result}
\setcounter{equation}{0} \label{sect-6} So far we have only proved a
local existence-uniqueness result for Equation \eqref{eqn:rho}. In
this section, we want to prove that, for any fixed $T>0$, the local
solution $\rho$ exists in the whole of $[0,T]$, at least for
sufficiently small value of $\e$. The main tool in this direction is
represented by the {\it a priori} estimates in the next subsection.

\subsection{A priori estimates}
\label{subsect-apriori}
The main result of this subsection is
contained in the following theorem.

\begin{theorem}
\label{cor-apriori-estim} For any $T>0$, there exist
$\varepsilon_{0}=\e_0(T)\in (0,1/2)$ and $K=K(T)>0$ such that, if $\rho\in {\mathscr Y}_T$
$($see Definition $\ref{def-YT})$ is a solution to Equation
\eqref{eqn:rho}, then
\begin{align*}
\sup_{{\tau\in (0,T]}\atop{\eta\in [-\ell_0/2,\ell_0/2]}}
|\rho_{\eta}&(\tau,\eta)| +\sup_{\tau\in (0,T]}
\int_{-\frac{\ell_0}{2}}^{\frac{\ell_0}{2}}(\rho_{\eta\eta}(\tau,\eta))^2d\eta\notag\\
&+\int_0^T\int_{-\frac{\ell_0}{2}}^{\frac{\ell_0}{2}}(\rho_{\tau}(\tau,\eta))^2
d\eta d\tau
+\int_0^T\int_{-\frac{\ell_0}{2}}^{\frac{\ell_0}{2}}(\rho_{\tau\eta}(\tau,\eta))^2
d\eta d\tau \leq K,
\end{align*}
for all $\tau\in(0,T]$, whenever $\varepsilon \leq\varepsilon_{0}$.
\end{theorem}

The proof of Theorem \ref{cor-apriori-estim} is obtained employing
an energy method. Let $\rho\in {\mathscr Y}_T$ solve
\eqref{eqn:rho}, i.e., the equation
\begin{equation}
\label{eqn:rho-1} \frac{\partial}{\partial \tau}{\mathscr
B}_{\e}(\rho) = {\mathscr S}_{\e}(\rho) -\Phi_\eta\rho_\eta
-\frac{\e}{2}(\rho_\eta)^2 + {\mathscr
G}_{\e}((\Phi_{\eta}+\e\rho_{\eta})^2)+{\mathscr H}_{\e}(\Phi),
\end{equation}
for some $T>0$. Multiplying both the sides of \eqref{eqn:rho-1} by
$\rho_{\tau}$ and integrating over $[-\ell_0/2,\ell_0/2]$, we get
\begin{align}
\int_{-\frac{\ell_0}{2}}^{\frac{\ell_0}{2}}{\mathscr
B}_{\e}(\rho_{\tau})\,\rho_{\tau}d\eta=&
\int_{-\frac{\ell_0}{2}}^{\frac{\ell_0}{2}}{\mathscr
S}_{\e}(\rho)\rho_{\tau}d\eta
-\int_{-\frac{\ell_0}{2}}^{\frac{\ell_0}{2}}\Phi_\eta\rho_\eta\rho_{\tau}d\eta
-\frac{\e}{2}\int_{-\frac{\ell_0}{2}}^{\frac{\ell_0}{2}}(\rho_\eta)^2 \rho_{\tau}d\eta\notag\\
& +\int_{-\frac{\ell_0}{2}}^{\frac{\ell_0}{2}}{\mathscr
G}_{\e}((\Phi_\eta+\e\rho_\eta)^2)\rho_{\tau}d\eta
+\int_{-\frac{\ell_0}{2}}^{\frac{\ell_0}{2}}{\mathscr
H}_{\e}(\Phi)\rho_{\tau}d\eta.
\label{int-by-parts}
\end{align}

Using the very definition of the operator ${\mathscr S}_{\e}$ (see
\eqref{S-ve}) and then integrating by parts, yields
\begin{align*}
&\int_{-\frac{\ell_0}{2}}^{\frac{\ell_0}{2}}{\mathscr S}_{\e}(\rho)\rho_{\tau}d\eta\\
=& \int_{-\frac{\ell_0}{2}}^{\frac{\ell_0}{2}}\left\{-3(1-\e)
\rho_{\eta\eta\eta\eta}-\rho_{\eta\eta}+3\e
\left(D_{\eta\eta}R(0,L_{\varepsilon})
[\rho_{\eta\eta}({\bf T}'-{\bf U})]\right)(\cdot,0,\cdot)\right\}\rho_{\tau}d\eta\\
=& -\frac{3}{2}(1-\e)
\frac{d}{dt}\int_{-\frac{\ell_0}{2}}^{\frac{\ell_0}{2}}(\rho_{\eta\eta})^2d\eta
- \int_{-\frac{\ell_0}{2}}^{\frac{\ell_0}{2}}\rho_{\eta\eta}\rho_{\tau}d\eta\\
& + 3\e\int_{-\frac{\ell_0}{2}}^{\frac{\ell_0}{2}}
\left(D_{\eta\eta}R(0,L_{\varepsilon})[ \rho_{\eta\eta}({\bf
T}'-{\bf U})]\right)(\cdot,0,\cdot)\rho_{\tau}d\eta.
\end{align*}
Therefore, we can write Equation \eqref{int-by-parts} in the following equivalent form:
\begin{align}
&\frac{3}{2}(1-\e)
\frac{d}{d\tau}\int_{-\frac{\ell_0}{2}}^{\frac{\ell_0}{2}}(\rho_{\eta\eta})^2d\eta
+ \int_{-\frac{\ell_0}{2}}^{\frac{\ell_0}{2}}{\mathscr B}_{\e}(\rho_\tau)\,\rho_\tau d\eta\notag\\
=&-\frac{\e}{2}\int_{-\frac{\ell_0}{2}}^{\frac{\ell_0}{2}}(\rho_\eta)^2\rho_{\tau}d\eta-
\int_{-\frac{\ell_0}{2}}^{\frac{\ell_0}{2}}\rho_{\eta\eta}\rho_{\tau}d\eta\notag\\
&  + 3\e\int_{-\frac{\ell_0}{2}}^{\frac{\ell_0}{2}}
\left(D_{\eta\eta}R(0,L_{\varepsilon})[
\rho_{\eta\eta}({\bf T}'-{\bf U})]\right)(\cdot,0,\cdot)\rho_{\tau}d\eta\notag\\
&+ \int_{-\frac{\ell_0}{2}}^{\frac{\ell_0}{2}}{\mathscr
G}_{\e}((\Phi_\eta+\e\rho_\eta)^2)\rho_{\tau}d\eta
-\int_{-\frac{\ell_0}{2}}^{\frac{\ell_0}{2}}\Phi_{\eta}\rho_{\eta}\rho_{\tau}d\eta
+\int_{-\frac{\ell_0}{2}}^{\frac{\ell_0}{2}}{\mathscr
H}_{\e}(\Phi)\rho_{\tau}d\eta. \label{stima-import}
\end{align}
In the following lemmata, we estimate the terms
\begin{align*}
&{\mathscr I}_1:=\int_{-\frac{\ell_0}{2}}^{\frac{\ell_0}{2}}{\mathscr
B}_{\e}(\rho_\tau)\,\rho_\tau d\eta,\\
&{\mathscr I}_2:=\int_{-\frac{\ell_0}{2}}^{\frac{\ell_0}{2}}
\left(D_{\eta\eta}R(0,L_{\varepsilon})[
\rho_{\eta\eta}({\bf T}'-{\bf U})]\right)(\cdot,0,\cdot)\rho_{\tau}d\eta,\\
&{\mathscr I}_3:=\int_{-\frac{\ell_0}{2}}^{\frac{\ell_0}{2}}{\mathscr G}_{\e}((\Phi_\eta+\e\rho_\eta)^2)\rho_{\tau}d\eta,\\
&{\mathscr
I}_4:=\int_{-\frac{\ell_0}{2}}^{\frac{\ell_0}{2}}{\mathscr
H}_{\e}(\Phi)\rho_{\tau}d\eta.
\end{align*}
The main issue is to control ${\mathscr I}_1$. We have the following
\begin{lemma}
\label{lem-1}
It holds that
\begin{eqnarray*}
{\mathscr I}_1(\tau) \geq
\int_{-\frac{\ell_0}{2}}^{\frac{\ell_0}{2}} (\rho(\tau,\cdot))^2 d\eta
+3\e\int_{-\frac{\ell_0}{2}}^{\frac{\ell_0}{2}}(\rho_{\eta}(\tau,\cdot))^2d\eta,
\end{eqnarray*}
for any $\tau\in [0,T]$ and any $\e\in (0,1/2]$.
\end{lemma}

\begin{proof}
Of course, we can limit ourselves to proving the estimate with
$\rho_{\tau}(\tau,\cdot)$ being replaced by $\varphi\in H^2$.

It is immediate to check that
\begin{eqnarray*}
\int_{-\frac{\ell_0}{2}}^{\frac{\ell_0}{2}}{\mathscr
B}_{\e}(\phi)\,\phi d\eta=
\sum_{k=0}^{+\infty}b_{\e,k}|\hat\varphi(k)|^2,
\end{eqnarray*}
where the symbol $(b_{\e,k})$ of the operator ${\mathscr B}_{\e}$ is
defined by \eqref{bek}. Note that $b_{\e,k}=h(X_{\e,k})$ for any
$k=0,1,\ldots$, where the function $h:[1,+\infty)\to\R$ is defined
by
\begin{eqnarray*}
h(s)=\frac{3}{4}\frac{(s+1)(s^2+2s-1)}{s+2},\qquad\;\,s\ge 1.
\end{eqnarray*}
Since $h(s)\ge (3s+1)/4$ for any $s\ge 1$, we can estimate
\begin{align*}
\int_{-\frac{\ell_0}{2}}^{\frac{\ell_0}{2}}{\mathscr
B}_{\e}(\phi)\,\phi d\eta\ge \sum_{k=0}^{+\infty}
(1+\e\lambda_k)|\hat\varphi(k)|^2
=&\int_{-\frac{\ell_0}{2}}^{\frac{\ell_0}{2}}|\varphi|^2d\eta
-3\e\int_{-\frac{\ell_0}{2}}^{\frac{\ell_0}{2}}\varphi\varphi_{\eta\eta}d\eta\\
=&\int_{-\frac{\ell_0}{2}}^{\frac{\ell_0}{2}}|\varphi|^2d\eta
+3\e\int_{-\frac{\ell_0}{2}}^{\frac{\ell_0}{2}}|\varphi_{\eta}|^2d\eta,
\end{align*}
and we are done.
\end{proof}

We now consider the terms ${\mathscr I}_2$, ${\mathscr I}_3$ and ${\mathscr I}_4$.

\begin{lemma}
\label{lem-2} For any $\tau\in [0,T]$ and any $\e\in (0,1/2]$, it
holds that
\begin{eqnarray*}
|{\mathscr I}_2(\tau)| \le
\frac{1}{12\e}\|\rho_{\tau}(\tau,\cdot)\|_2\|\rho_{\eta\eta}(\tau,\cdot)\|_2.
\end{eqnarray*}
\end{lemma}

\begin{proof}
As it is immediately seen, for any $\tau\in [0,T]$ we can estimate
\begin{align*}
&\left |\int_{-\frac{\ell_0}{2}}^{\frac{\ell_0}{2}}
\left(D_{\eta\eta}R(0,L_{\varepsilon})[ \rho_{\eta\eta}({\bf
T}'-{\bf U})]\right)(\tau,0,\cdot)\rho_{\tau}(\tau,\cdot)d\eta\right |\\
\le& \frac{1}{3\e}\|\varrho_{\tau}(\tau,\cdot)\|_2
\|\left(D_{\eta\eta}R(0,L_{\varepsilon})[ \rho_{\eta\eta}({\bf
T}'-{\bf U})]\right)(\tau,0,\cdot)\|_2.
\end{align*}
To compute the $L^2$-norm of the function
$\left(D_{\eta\eta}R(0,L_{\varepsilon})[ \rho_{\eta\eta}({\bf
T}'-{\bf U})]\right)(\tau,0,\cdot)$, we take advantage of Formula
\eqref{u2k}, which allows us to estimate
\begin{align*}
\|\left(D_{\eta\eta}R(0,L_{\varepsilon})[ \rho_{\eta\eta}({\bf
T}'-{\bf U})]\right)(\tau,0,\cdot)\|_2^2 =&
\frac{1}{4}\sum_{k=0}^{+\infty}\left
|\frac{X_k+1}{X_k+2}\lambda_k^2\hat\varphi(\tau,k)\right |^2\\
\le &\frac{1}{4}\sum_{k=0}^{+\infty}\lambda_k^2|\hat\varphi(\tau,k)|^2
=\frac{1}{4}\|\varphi_{\eta\eta}(\tau,\cdot)\|_2^2,
\end{align*}
where, as usual, $X_k=\sqrt{1+4\e\lambda_k}$ and $\hat\rho(\tau,k)$ is the $k$-th Fourier coefficient of the function
$\rho(\tau,\cdot)$. This accomplishes the
proof.
\end{proof}
\begin{lemma}
\label{lem-3} There exists a positive constant $C$, independent of
$\e\in (0,1/2]$ and $\tau\in [0,T]$, such that
\begin{align*}
|{\mathscr I}_3(\tau)|\le&
C\Big (\|\rho_{\tau}(\tau,\cdot)\|_2+\|\rho_{\tau}(\tau,\cdot)\|_2\|\rho_{\eta\eta}(\tau,\cdot)\|_2
+\e\|\rho_{\tau}(\tau,\cdot)\|_2\|\rho_{\eta\eta}(\tau,\cdot)\|_2^2\notag\\
&\qquad\; +\e\|\rho_{\tau\eta}(\tau,\cdot)\|_2\|\rho_{\eta\eta}(\tau,\cdot)\|_2
+\e^2\|\rho_{\tau\eta}(\tau,\cdot)\|_2\|\rho_{\eta\eta}(\tau,\cdot)\|_2^2\Big ),
\end{align*}
for any $\tau\in [0,T]$.
\end{lemma}

\begin{proof} As in the proof of the previous lemma, it is enough to
estimate the $L^2$-norm of the function ${\mathscr
G}_{\e}((\Phi_\eta(\tau,\cdot)+\e\rho_\eta(\tau,\cdot))^2)$. For
this purpose, we observe that we can estimate the $L^2$-norm of the
function $ {\mathscr G}_{\e}(\psi)$, for any $\psi\in H^2$, by
\begin{align*}
\|{\mathscr G}_{\e}(\psi)\|_2^2=&\sum_{k=0}^{+\infty}
\lambda_k^2\left (\frac{3X_k^2+15X_k+4}{2(X_k+1)(X_k+2)}\right
)^2|\hat\psi(k)|^2\\
 \le &\frac{121}{36}\sum_{k=0}^{+\infty}\lambda_k^2|\hat\psi(k)|^2
\le 4\|\psi_{\eta\eta}\|_2^2,
\end{align*}
where $X_k=\sqrt{1+4\e\lambda_k}$ for any $k=0,1,\ldots$ It follows
that
\begin{equation}
\|{\mathscr G}_{\e}(\psi)\|_2\le 2\|\psi_{\eta\eta}\|_2.
\label{form-1}
\end{equation}
Moreover, the symbol $g_k$ can be split as follows:
\begin{eqnarray*}
g_k=-\frac{3}{2}\lambda_k+\frac{1}{4\e}h(X_k),\qquad\;\,k=0,1,\ldots,
\end{eqnarray*}
where the function $h:[1,+\infty)\to\R$ is defined by
\begin{eqnarray*}
h(s)=\frac{(3s-1)(s-1)}{s^2+2},\qquad\;\,s\ge 1.
\end{eqnarray*}
Clearly, $0\le h(s)\le
1$ for any $s\ge 1$. Hence, we can split
\begin{equation}
{\mathscr G}_{\e}(\psi)=\frac{3}{2}\psi_{yy}+
\frac{1}{4\e}{\mathscr G}_{\e}^{(1)}(\psi), \label{form-2}
\end{equation}
where the operator ${\mathscr G}_{\e}^{(1)}$ is well-defined in
$L^2$ and
\begin{equation}
\|{\mathscr G}_{\e}^{(1)}(\psi)\|_2\le 3\|\psi\|_2.
\label{form-3}
\end{equation}

We now split (for any arbitrarily fixed $\tau\in [0,T]$)
\begin{align*}
&\int_{-\frac{\ell_0}{2}}^{\frac{\ell_0}{2}}{\mathscr
G}_{\e}((\Phi_\eta(\tau,\cdot)+\e\rho_\eta(\tau,\cdot))^2)\rho_{\tau}(\tau,\cdot)d\eta\\
=&\int_{-\frac{\ell_0}{2}}^{\frac{\ell_0}{2}}{\mathscr
G}_{\e}((\Phi_\eta(\tau,\cdot))^2)\rho_{\tau}(\tau,\cdot)d\eta+2\e\int_{-\frac{\ell_0}{2}}^{\frac{\ell_0}{2}}{\mathscr
G}_{\e}(\Phi_\eta(\tau,\cdot)\rho_\eta(\tau,\cdot))\rho_{\tau}(\tau,\cdot)d\eta\\
&+\e^2\int_{-\frac{\ell_0}{2}}^{\frac{\ell_0}{2}}{\mathscr
G}_{\e}((\rho_\eta(\tau,\cdot))^2)\rho_{\tau}d\eta
:= J_1(\tau)+J_2(\tau)+J_3(\tau).
\end{align*}

To estimate $J_1$, we use Formula \eqref{form-1} and H\"older
inequality to get
\begin{equation}
|J_1(\tau)|\le
2\|\rho_{\tau}(\tau,\cdot)\|_2\|((\Phi_{\eta}(\tau,\cdot))^2)_{\eta\eta}\|_2.
\label{estim-I1}
\end{equation}
Estimating the terms $J_2$ and $J_3$ is a bit more tricky. Using
Formulae \eqref{form-2} and \eqref{form-3}, we get
\begin{align}
|J_2(\tau)|\le& 3\e\left |\int_{-\frac{\ell_0}{2}}^{\frac{\ell_0}{2}}
\rho_{\tau}(\tau,\cdot)(\Phi_{\eta}(\tau,\cdot)\rho_{\eta}(\tau,\cdot))_{\eta\eta}d\eta\right | +
\frac{3}{2}\|\Phi_{\eta}(\tau,\cdot)\rho_{\eta}(\tau,\cdot)\|_2\|\rho_{\tau}(\tau,\cdot)\|_2\notag\\
=&3\e\left |\int_{-\frac{\ell_0}{2}}^{\frac{\ell_0}{2}}
\rho_{\tau\eta}(\tau,\cdot)(\Phi_{\eta\eta}(\tau,\cdot)\rho_{\eta}(\tau,\cdot)+\Phi_{\eta}(\tau,\cdot)\rho_{\eta\eta}(\tau,\cdot))
d\eta\right|\notag\\
&+ \frac{3}{2}\|\Phi_{\eta}(\tau,\cdot)\rho_{\eta}(\tau,\cdot)\|_2\|\rho_{\tau}(\tau,\cdot)\|_2\notag\\
\le&
3\e\|\Phi_{\eta\eta}\|_{\infty}\|\rho_{\tau\eta}(\tau,\cdot)\|_2\|\rho_{\eta}(\tau,\cdot)\|_2
+3\e\|\Phi_{\eta}\|_{\infty}\|\rho_{\tau\eta}(\tau,\cdot)\|_2\|\rho_{\eta\eta}(\tau,\cdot)\|_2\notag\\
&+\frac{3}{2}\|\Phi_{\eta}\|_{\infty}\|\rho_{\eta}(\tau,\cdot)\|_2\|\rho_{\tau}(\tau,\cdot)\|_2.
\label{estim-I2-0}
\end{align}
Using a Poincar\'e-Wirtinger inequality, we can continue Estimate \eqref{estim-I2} and obtain that
\begin{equation}
|J_2(\tau)|\le C_1\left (\e\|\rho_{\tau\eta}(\tau,\cdot)\|_2\|\rho_{\eta\eta}(\tau,\cdot)\|_2+
\|\rho_{\eta\eta}(\tau,\cdot)\|_2\|\rho_{\tau}(\tau,\cdot)\|_2\right ). \label{estim-I2}
 \end{equation}
To estimate the term $I_3(\tau)$, we can argue similarly. Therefore,
\begin{align}
|J_3(\tau)|\le& \frac{3}{2}\e^2\left
|\int_{-\frac{\ell_0}{2}}^{\frac{\ell_0}{2}}\rho_{\tau}(\tau,\cdot)((\rho_{\eta}(\tau,\cdot))^2)_{\eta\eta}d\eta\right
|+\frac{3}{4}\e\int_{-\frac{\ell_0}{2}}^{\frac{\ell_0}{2}}|(\rho_{\eta}(\tau,\cdot))^2\rho_{\tau}(\tau,\cdot)|d\eta\notag\\
=&\frac{3}{2}\e^2\left
|\int_{-\frac{\ell_0}{2}}^{\frac{\ell_0}{2}}\rho_{\tau\eta}(\tau,\cdot)((\rho_{\eta}(\tau,\cdot))^2)_{\eta}d\eta\right
|+\frac{3}{4}\e\int_{-\frac{\ell_0}{2}}^{\frac{\ell_0}{2}}|(\rho_{\eta}(\tau,\cdot))^2\rho_{\tau}(\tau,\cdot)|d\eta\notag\\
\le&
3\e^2\|\rho_{\tau\eta}(\tau,\cdot)\|_2\|\rho_{\eta}(\tau,\cdot)\|_{\infty}\|\rho_{\eta\eta}(\tau,\cdot)\|_2
+\frac{3}{4}\e\|(\rho_{\eta}(\tau,\cdot))^2\|_2\|\rho_{\tau}(\tau,\cdot)\|_2\notag\\
\le& 3\e^2\sqrt{\ell_0}\|\rho_{\tau\eta}(\tau,\cdot)\|_2\|\rho_{\eta\eta}(\tau,\cdot)\|_2^2
+\frac{3}{4}\ell_0^{\frac{3}{2}}\e\|\rho_{\eta\eta}(\tau,\cdot)\|_2^2\|\rho_{\tau}(\tau,\cdot)\|_2.
\label{estim-I3}
\end{align}

Combining Estimates \eqref{estim-I1}, \eqref{estim-I2-0} and
\eqref{estim-I3} together, the assertion follows at once.
\end{proof}

\begin{lemma}
\label{lem-4} There exists a positive constant $C$, independent of
$\e\in (0,1/2]$ and $\tau\in [0,T]$, such that
\begin{align}
|{\mathscr I}_4(\tau)|\le&
C\Big (\|\Phi_{\tau\eta\eta}(\tau,\cdot)\|_2+\|\Phi_{\eta\eta\eta\eta}(\tau,\cdot)\|_2\Big )
\|\rho_{\tau}(\tau,\cdot)\|_2,\qquad\tau\in [0,T].
\label{stima-I4}
\end{align}
\end{lemma}

\begin{proof}
Of course, we just need to estimate the terms
\begin{align*}
J_4^{(1)}(\tau)&=\e\int_{-\frac{\ell_0}{2}}^{\frac{\ell_0}{2}}
\rho_{\tau}\left (D_{\eta\eta}R(0,L_{\varepsilon}) \left [\Phi_{\tau\eta\eta}\left
({\bf V}-{\bf T}-\frac{4}{3}{\bf U}\right )\right ]\right
)(\cdot,0,\cdot)d\eta,\\[2mm]
J_4^{(2)}(\tau)&=\int_{-\frac{\ell_0}{2}}^{\frac{\ell_0}{2}}
\rho_{\tau}\left(D_{\eta\eta}R(0,L_{\varepsilon})[ \Phi_{\eta\eta}({\bf
T}'-{\bf U})]\right)(\cdot,0,\cdot),
\end{align*}
the other remaining terms are easily to be handled with.

Concerning $J_4^{(1)}$, taking \eqref{u1k} into account we can
estimate
\begin{align*}
|J_4^{(1)}(\tau)|&\le \frac{4}{9}\e\sum_{k=0}^{+\infty}
\lambda_k^2\frac{4X_k+7}{(X_k+1)^2(X_k+2)}|\hat\rho_{\tau}(\tau,k)|
|\hat\Phi_{\tau}(\tau,k)|\\
&\le \frac{1}{9}\sum_{k=0}^{+\infty}
\lambda_k\frac{X_k^2(4X_k+7)}{(X_k+1)^2(X_k+2)}|\hat\rho_{\tau}(\tau,k)|
|\hat\Phi_{\tau}(\tau,k)|\\
&\le \frac{4}{9}\sum_{k=0}^{+\infty}
\lambda_k|\hat\rho_{\tau}(\tau,k)|
|\hat\Phi_{\tau}(\tau,k)|\\
&=\frac{4}{9}\|\rho_{\tau}(\tau,\cdot)\|_2\|\Phi_{\tau\eta\eta}(\tau,\cdot)\|_2.
\end{align*}
for any $\tau\in [0,T]$. Similarly, using \eqref{u2k} we can estimate
\begin{align*}
|J_4^{(2)}(\tau)|&\le \frac{2}{3}\sum_{k=0}^{+\infty}
\lambda_k^2\frac{1}{(X_k+1)(X_k+2)}|\hat\rho_{\tau}(\tau,k)|
|\hat\Phi(\tau,k)|\\
&\le \frac{1}{3}\sum_{k=0}^{+\infty}
\lambda_k^2|\hat\rho_{\tau}(\tau,k)|
|\hat\Phi(\tau,k)|\\
&= \frac{1}{3}\|\rho_{\tau}(\tau,\cdot)\|_2
|\Phi_{\eta\eta\eta\eta}(\tau,\cdot)\|_2,
\end{align*}
for any $\tau\in [0,T]$. Now estimate \eqref{stima-I4} follows immediately.
\end{proof}

Finally, we recall the following result proved in \cite{BFHLS} which
plays a crucial role in the proof of Theorem
\ref{cor-apriori-estim}.

\begin{lemma}[slight extension of Lemma 3.1 of \cite{BFHLS}]
\label{lem-5} Let $A_0$, $C_0$, $C_1$, $C_2$ be positive constants.
For any $T>0$, there exist $\e_0\in (0,1/2)$ and
a constant $K_0$ such that, if $A_{\e}\in C^1([0,T'])$ $(T'\in (0,T])$
satisfies
\begin{align*}
\left\{
\begin{array}{ll}
A_{\e}'(\tau)\le C_0+C_1A_{\e}(\tau)+C_2\e (A_{\e}(\tau))^2, & \tau\in [0,T'],\\[1mm]
A_{\e}(0)\le A_0,
\end{array}
\right.
\end{align*}
for some $\e\in (0,\e_0]$, then $A_{\e}(\tau)\le K_0$ for any
$\tau\in [0,T_0]$.
\end{lemma}

\begin{proof}[Proof of Theorem $\ref{cor-apriori-estim}$]
To begin with, we observe that, taking Poincar\'e-Wirtinger
inequality into account, we can estimate
\begin{align*}
\left
|\int_{-\frac{\ell_0}{2}}^{\frac{\ell_0}{2}}(\rho_\eta(\tau,\cdot))^2\rho_{\tau}(\tau,\cdot)d\eta
\right |&\le
\|\rho_{\tau}(\tau,\cdot)\|_2\|\rho_{\eta}(\tau,\cdot)\|_2\|\rho_{\eta}(\tau,\cdot)\|_{\infty}\notag\\
&\le
\sqrt{\ell_0}\|\rho_{\tau}(\tau,\cdot)\|_2\|\rho_{\eta\eta}(\tau,\cdot)\|_2^2,
\end{align*}
for any $\tau\in [0,T]$. Hence, from Lemmata \ref{lem-1}-\ref{lem-4}
and estimate \eqref{stima-import}, and using H\"older inequality, we
get
\begin{align}
&\frac{3}{2}(1-\e)\frac{d}{d\tau}\|\rho_{\eta\eta}(\tau,\cdot)\|_2^2
+\|\rho_{\tau}(\tau,\cdot)\|_2^2+3\e\|\rho_{\tau\eta}(\tau,\cdot)\|_2^2\notag\\
 \le& C\Big (
\|\rho_{\tau}(\tau,\cdot)\|_2+\|\rho_{\tau}(\tau,\cdot)\|_2\|\rho_{\eta\eta}(\tau,\cdot)\|_2
+\e\|\rho_{\tau}(\tau,\cdot)\|_2\|\rho_{\eta\eta}(\tau,\cdot)\|_2^2\notag\\
&\quad\;\,+\e\|\rho_{\tau\eta}(\tau,\cdot)\|_2\|\rho_{\eta\eta}(\tau,\cdot)\|_2
+\e^2\|\rho_{\tau\eta}(\tau,\cdot)\|_2\|\rho_{\eta\eta}(\tau,\cdot)\|_2^2\Big
),
\label{stima-imp-1}
\end{align}
for some positive constant $C$, depending on $\Phi$ but being
independent of $\tau\in [0,T]$. Using Young inequality
$ab\le\frac{1}{4}a^2+b^2$, we can estimate
\begin{align*}
&C\|\rho_{\tau}(\tau,\cdot)\|_2\le\frac{1}{4}\|\rho_{\tau}(\tau,\cdot)\|_2^2+C^2,\\[1mm]
&C\|\rho_{\tau}(\tau,\cdot)\|_2\|\rho_{\eta\eta}(\tau,\cdot)\|_2\le
\frac{1}{4}\|\rho_{\tau}(\tau,\cdot)\|_2^2+C^2\|\rho_{\eta\eta}(\tau,\cdot)\|_2^2,\\[1mm]
&C\e\|\rho_{\tau}(\tau,\cdot)\|_2\|\rho_{\eta\eta}(\tau,\cdot)\|_2^2
\le\frac{1}{4}\|\rho_{\tau}(\tau,\cdot)\|^2_2+C^2\e^2\|\rho_{\eta\eta}(\tau,\cdot)\|_2^4,\\[1mm]
&C\e\|\rho_{\tau\eta}(\tau,\cdot)\|_2\|\rho_{\eta\eta}(\tau,\cdot)\|_2
\le\frac{1}{4}\e\|\rho_{\tau\eta}(\tau,\cdot)\|_2^2
+\e C^2\|\rho_{\eta\eta}(\tau,\cdot)\|_2^2,\\[1mm]
&C\e^2\|\rho_{\tau\eta}(\tau,\cdot)\|_2\|\rho_{\eta\eta}(\tau,\cdot)\|_2
\le\frac{1}{4}\e^2\|\rho_{\tau\eta}(\tau,\cdot)\|_2^2
+C^2\e^2\|\rho_{\eta\eta}(\tau,\cdot)\|_2^2.
\end{align*}
Hence, from \eqref{stima-imp-1} we get
\begin{align}
&\frac{3}{2}(1-\e)\frac{d}{d\tau}\|\rho_{\eta\eta}(\tau,\cdot)\|_2^2
+\|\rho_{\tau}(\tau,\cdot)\|_2^2+3\e\|\rho_{\tau\eta}(\tau,\cdot)\|_2^2\notag\\[1mm]
 \le&
C^2+\frac{3}{4}\|\rho_{\tau}(\tau,\cdot)\|_2^2+\frac{1}{4}(\e+\e^2)\|\rho_{\tau\eta}(\tau,\cdot)\|_2^2
\notag\\[1mm]
&+C^2(1+\e+\e^2)\|\rho_{\eta\eta}(\tau,\cdot)\|_2^2+C^2\e^2\|\rho_{\eta\eta}(\tau,\cdot)\|_2^4,
\label{giro-0}
\end{align}
or, equivalently,
\begin{align}
\frac{d}{d\tau}\|\rho_{\eta\eta}(\tau,\cdot)\|_2^2&\le
\frac{4}{3}C^2
+4C^2\|\rho_{\eta\eta}(\tau,\cdot)\|_2^2+\frac{4}{3}C^2\e^2\|\rho_{\eta\eta}(\tau,\cdot)\|_2^4\notag\\[1mm]
&\le \frac{4}{3}C^2
+4C^2\|\rho_{\eta\eta}(\tau,\cdot)\|_2^2+\frac{2}{3}C^2\e\|\rho_{\eta\eta}(\tau,\cdot)\|_2^4,
\label{giro}
\end{align}
provided that $\e\le 1/2$.

Applying Lemma \ref{lem-5} to \eqref{giro} with
$A_{\e}(\tau)=\|\rho_{\eta\eta}(\tau,\cdot)\|_2^2$ and
$(C_0,C_1,C_2)=(4C^2/3,4C^2,2C^2/3)$, we immediately deduce that
there exist $\e_0\in (0,1/2)$ and $K_0>0$ such that
\begin{eqnarray*}
\sup_{\tau\in (0,T]}
\int_{-\frac{\ell_0}{2}}^{\frac{\ell_0}{2}}(\rho_{\eta\eta}(\tau,\eta))^2d\eta
\le K_0,
\end{eqnarray*}
for any $\e\in (0,\e_0)$. Now, using a Poincar\'e-Wirtinger
inequality, we get
\begin{eqnarray*}
\sup_{{\tau\in (0,T]}\atop{\eta\in [-\ell_0/2,\ell_0/2]}}
|\rho_{\eta}(\tau,\eta)| \le K_1,
\end{eqnarray*}
for any $\e\in (0,\e_0]$, with $K_1$ independent of $\e$.

Finally, integrating \eqref{giro-0} and using the estimates so far
obtained, we deduce that
\begin{eqnarray*}
\int_0^{T_0}\int_{-\frac{\ell_0}{2}}^{\frac{\ell_0}{2}}(\rho_{\tau}(\tau,\eta))^2d\tau
d\eta
+\int_0^{T_0}\int_{-\frac{\ell_0}{2}}^{\frac{\ell_0}{2}}(\rho_{\tau\eta}(\tau,\eta))^2d\tau
d\eta \le K_2,
\end{eqnarray*}
for some constant $K_2$, independent of $\e$. The assertion now
follows.
\end{proof}

\begin{cor}
\label{final} Under the assumptions of Theorem
$\ref{cor-apriori-estim}$, there exists a constant $M>0$ such that
\begin{eqnarray*}
\|\rho\|_{C^{0,1}([0,T]\times [-\ell_0/2,\ell_0/2])} \leq M,
\end{eqnarray*}
for any $\e\in (0,\e_0]$.
\end{cor}

\begin{proof}
In view of Theorem \ref{cor-apriori-estim}, we have only to estimate
the sup-norm of the function $\rho$. Since $\rho\in {\mathscr Y}_T$
and $\rho(0,\cdot)=0$, we have
\begin{eqnarray*}
\sup_{\tau\in
[0,T]}|\rho(\tau,\eta)|\le\int_0^T|\rho_{\tau}(\sigma,\eta)|d\sigma,\qquad\;\,\eta\in
[-\textstyle{\frac{\ell_0}{2}},\frac{\ell_0}{2}].
\end{eqnarray*}

Integrating both the sides of the previous inequality with respect
to $\eta\in [-\ell_0/2,\ell_0/2]$, we get
\begin{align*}
\int_{-\frac{\ell_0}{2}}^{-\frac{\ell_0}{2}}\sup_{\tau\in
[0,T]}|\rho(\tau,\eta)|d\eta \le&
\int_{-\frac{\ell_0}{2}}^{-\frac{\ell_0}{2}}d\eta\int_0^{T}|\rho_{\tau}(\sigma,\eta)|d\sigma\notag\\
\le&
\sqrt{\ell_0T}\int_0^{T}d\sigma\int_{-\frac{\ell_0}{2}}^{-\frac{\ell_0}{2}}|\rho_{\tau}(\sigma,\eta)|^2d\eta\le
K\sqrt{\ell_0T},
\end{align*}
where $K$ is the constant in Theorem \ref{cor-apriori-estim}.
Therefore, the function $\rho$ remains in a bounded
subset of the space $L^1((-\ell_0/2,\ell_0/2);L^{\infty}(0,T))$.
Thanks to the uniform estimate on $\rho_{\eta}$ on $[0,T] \times
[-\ell_0/2,\ell_0/2]$,  we infer that $\rho$ is bounded
in $W^{1,1}((-\ell_0/2,\ell_0/2);L^{\infty}(0,T))$. Hence, by the
Sobolev embedding, $\rho$ is bounded in
$[0,T]\times [-\ell_0/2,\ell_0/2]$. This accomplishes the proof.
\end{proof}

\subsection{Solving Equation \eqref{eqn:rho} in $[0,T]$}
\label{subsect-6.2} We now consider a fixed time interval $[0,T]$
and $0<\e \leq \e_0$, with $\e_0=\e_0(T)$ given by Theorem
\ref{cor-apriori-estim}. Thanks to the a priori estimates of
Subsection \ref{subsect-apriori} and a classical result for
semilinear problems, we can show that, for any $\e\in (0,\e_0]$, the
solution $\rho=\rho_{\e}$ to Problem \eqref{eqn:rho}, given by
Theorem \ref{thm-semilinear}, can be extended with a function
$\rho\in {\mathscr Y}_{T}$ (see Definition \ref{def-YT}), which
solves the equation in the whole of $[0,T]$.

\begin{theorem}
Fix $T>0$ and let $\e_0=\e_0(T_0)$ be as in Theorem
$\ref{cor-apriori-estim}$. Then, for any $\e\in (0,\e_0]$, Equation
\eqref{eqn:rho} admits a unique solution $\rho\in {\mathscr Y}_T$.
\end{theorem}
\begin{proof}
Let us fix $T$ as in the statement of the theorem and let $\e\in
(0,\e_0)$. Suppose by contradiction that $T_{\e}<T$. Then, by
Theorem \ref{cor-apriori-estim},
\begin{eqnarray*}
\sup_{\tau\in
[0,T_{\e})}\|\rho_{\eta}(\tau,\cdot)\|_{C^1([-\ell_0/2,\ell_0/2])}<K,
\end{eqnarray*}
for some positive constant $K$, independent of $\e$. Hence, the
function ${\mathscr H}_{\e}(\cdot,\rho_{\eta}^2)$ is bounded in
$[0,T_{\e})\times [-\ell_0/2,\ell_0/2]$. In view of \cite[Prop.
7.1.8]{lunardi}, applied to Equation \eqref{eqn:rho}, this leads us
to a contradiction.
\end{proof}
\subsection{Proof of Main Theorem}
\label{subsect-6.3} We are now in a position to prove the main
result of this paper. Let us fix a function $\Phi_0\in
C_{\sharp}^{4+4\beta}$ for some $\beta\in (1/2,1)$.

From the results in Subsection \ref{subsect-6.2}, we know that, for
any $T>0$, there exists $\e_0=\e_0(T)$ such that Equation
\eqref{abstract} admits a unique solution $\psi_{\e}\in {\mathscr
Y}_T$ (see Definition \ref{def-YT}) such that
$\rho(0,\cdot)=\Phi_0$. Moreover, by Corollary \ref{final},
\begin{eqnarray*}
\|\psi_{\e}(\tau,\cdot)-\Phi_0(\tau,\cdot)\|_{C([-\ell_0/2,\ell_0/2])}\le \varepsilon M,\qquad\;\,\tau\in [0,T],
\end{eqnarray*}
for some positive constant $M$ and any $\e\in (0,\e_0]$.

In view of Theorem \ref{thm-equivalence}, there exists a (unique)
function $v\in {\mathscr V}_T$ such that the pair $(v,\psi)$ is the
unique solution to Problem \eqref{eqn:v}-\eqref{cond-v-2}.

Coming back to Problem \eqref{eqn:u}-\eqref{cond-u-2} and setting
$\ell_{\e}=\ell_0/\sqrt{\e}$ and $T_{\e}=T/\e^{2}$, it is now
immediate to conclude that, for any $\e\in (0,\e_0]$, it admits a
unique solution $(u,\varphi)\in {\mathscr V}_{T_{\e}}\times
{\mathscr Y}_{T_{\e}}$. Moreover,
\begin{eqnarray*}
\|\varphi_{\e}(t,\cdot)-\e\Phi(t\e^2,\sqrt{\e}\,\cdot)\|_{C([-\ell_{\e}/2,\ell_{\e}/2])}\le
\varepsilon^2 M,\qquad\;\,t\in [0,T_{\e}].
\end{eqnarray*}
This accomplishes the proof of Main Theorem.

\section*{Acknowledgments}
J.H. and L.L. were visiting professors at the University of Bordeaux 1
respectively in 2006-2007 and 2007-2009. They greatly acknowledge
the Institute of Mathematics of Bordeaux for the
warm hospitality during their visits.

\appendix

\section{Some results from \cite{BHL08}}
\setcounter{equation}{0}
In this appendix we recall some results
from \cite{BHL08} that are used throughout this paper.

\subsection{The operator $L$}
\label{Lstuff} Let ${\mathscr L}$ be the differential operator
defined on smooth (generalized) functions ${\bf u}$ by
\begin{eqnarray*}
({\mathscr L}{\bf u})(x,\eta)= \left\{
\begin{array}{lll}
D_{xx}u_1(x,\eta)-D_xu_1(x,\eta)+e^xu_1(0,\eta), &x\le 0, &|\eta|\le \frac{\ell_0}{2},\\[2mm]
D_{xx}u_2(x,\eta)-D_xu_2(x,\eta), &x\ge 0, &|\eta|\le
\frac{\ell_0}{2},
\end{array}
\right.
\end{eqnarray*}
and let $L$ be its the realization in ${\mathscr X}$, defined by
\begin{eqnarray*}
\left\{
\begin{array}{l}
D({L})= \Big\{{\bf u}\in C^{2,0}(I_-)\times C^{2,0}(I_+):
{\bf u},~{\mathscr L}{\bf u}\in {\mathscr X}\\[1mm]
\qquad\qquad\quad D_x^{(j)}u_1(0,\cdot)=D_x^{(j)}u_2(0,\cdot),~j=0,1\Big\},\\[3mm]
{L}{\bf u}= \left\{
\begin{array}{ll}
D_{xx}u_1-D_xu_1 +u_1(0,\cdot)e^x, & (x,\eta)\in I_-,\\[1mm]
D_{xx}u_2-D_xu_2, & (x,\eta)\in I_+.
\end{array}
\right.
\end{array}
\right.
\end{eqnarray*}

\begin{theorem}
\label{thm:1} The following properties are met:
\begin{enumerate}[\rm (i)]
\item
the operator ${L}$ is sectorial and, hence, it generates an analytic
semigroup in ${\mathscr X}$;
\item
the spectrum of the operator ${L}$ consists of $0$ and the halfline
$(-\infty,-1/4]$;
\item
the spectral projection on the kernel of ${L}$ is the operator
${\mathscr P}$ defined by
\begin{eqnarray*}
\qquad\;\;\;\;{\mathscr P}({\bf f})=\left
(\int_{-\infty}^0f_1(x,\cdot)dx+\int_0^{+\infty}e^{-x}f_2(x,\cdot)dx\right
) {\bf U}:= Q({\bf f}){\bf U},\qquad\;\,{\bf f}\in {\mathscr X};
\end{eqnarray*}
\item
let ${\bf f}\in {\mathscr X}$. Then, the equation ${L}{\bf u}={\bf
f}$ has a solution ${\bf u}\in D({L})$ if and only if ${\mathscr
P}({\bf f})=0$;
\item
for any $\lambda\notin (-\infty,-1/4]\cup\{0\}$ and any ${\bf
f}=(f_1,f_2)\in {\mathscr X}$, setting ${\bf u}:=R(\lambda,L){\bf
f}$ it holds that
\begin{align}
\qquad u_1(0,\eta)=u_2(0,\eta)=& g(\lambda)\left
(\int_{-\infty}^0e^{-\nu_1t}f_1(t,\eta)dt
+\int_0^{+\infty}e^{-\nu_2t}f_2(t,\eta)dt\right ), \label{3.1-b}
\end{align}
for any $\eta\in [-\ell_0/2,\ell_0/2]$, where
\begin{align*}
g(\lambda)&=\left (\frac{2\lambda
}{1+(2\lambda-1)X(\lambda)}\frac{1}{\nu_2}+1\right
)\frac{1}{X(\lambda)},\qquad\;\,X(\lambda)=\sqrt{1+4\lambda}.
\end{align*}
\end{enumerate}
\end{theorem}

\subsection{The operator $L_{\eps}$.}
For any $\eps>0$, we consider the operator $L+\eps A$ defined by
\begin{eqnarray*}
\left\{
\begin{array}{l}
D(L+\eps A)= \Big\{{\bf u}\in \tilde u_1\in C^{2,0}(I_-)
\cap C^{0,2}(I_-)\times C^{2,0}(I_+)\cap C^{0,2}(I_+):\\[1mm]
\qquad\qquad\qquad\;\;\; {\bf u},~{\bf u}_{\eta\eta},
~{\mathscr L}{\bf u}\in {\mathscr X},~D_x^{(j)}u_1(0,\cdot)=D_x^{(j)}u_2(0,\cdot),~j=0,1,\\[1mm]
\qquad\qquad\qquad\;\;\;
D^{(j)}_{\eta}u_i(\cdot,-\ell_0/2)=D^{(j)}_{\eta}u_i(\cdot,\ell_0/2),
~i=1,2,~j=0,1\Big\},\\[3.5mm]
({L}+\eps A){\bf u}= \left\{
\begin{array}{ll}
D_{xx} u_1+\eps D_{\eta\eta}u_1-D_xu_1 +u_1(0,\cdot)e^x, & (x,y)\in I_-,\\[1mm]
D_{xx}u_2+\eps D_{\eta\eta}u_2-D_xu_2, & (x,y)\in I_+.
\end{array}
\right.
\end{array}
\right.
\end{eqnarray*}

\begin{theorem}
\label{thm-generation} The following properties are met.
\begin{enumerate}[\rm (i)]
\item
The operator $L+\eps A$ is closable and its closure $L_{\eps}$ is
sectorial;
\item
the restriction of $L_{\eps}$ to $(I-{\mathscr P})({\mathscr X})$ is
sectorial and $0$ is in its resolvent set;
\item
let ${\bf f}={\bf h}\,\varphi$ for some ${\bf h}\in (I-{\mathscr
P})({\mathscr X})$, independent of $y$, and some $\varphi\in
C_{\sharp}^{2\alpha}$ $(\alpha\in (0,1)\setminus\{1/2\})$. Then, the
function $R(0,L_{\eps}){\bf f}$ belongs to $D(L+\e A)$. Moreover,
there exists a positive constant $C$, depending on $\eps$ and
$\alpha$ but being independent of ${\bf h}$ and $\varphi$, such that
\begin{equation}
\qquad\;\;\;\|D_{x}^{(i)}R(0,L_{\eps}){\bf f}\|_{\mathscr
X}+\|D_{\eta}^{(i)}R(0,L_{\eps}){\bf f}\|_{\mathscr X} \le C\|{\bf
h}\|_{\mathscr X}\|\varphi\|_{C^{2\alpha}([-\ell_0/2,\ell_0/2])},
\label{estim-R}
\end{equation}
for $i=0,1,2$.
\end{enumerate}
\end{theorem}

\subsection{Proof of Theorem \ref{thm-4.1}}

We split the proof into three steps. In the first one, we show that
Problem \eqref{4order} admits a unique solution $\Phi$ in some time
domain  $[0,T_0]$. Since this result can be proved using the same
arguments as in Subsection \ref{subsect-proof}, we just sketch the
proof. Then, in Steps 2 and 3, we show that $\Phi$ exists and is
smooth in the whole of $[0,+\infty)$.
\par
{\em Step 1.} As we have already remarked in the proof of
Proposition \ref{prop-5.4}, the realization $A$ of the second order
derivative in $C([-\ell_0/2,\ell_0/2])$, with domain
$C_{\sharp}^1\cap C^2$, is a sectorial operator with spectrum
contained in $(-\infty,0]$. By \cite[Prop. 2.4.1 \& 2.4.4]{lunardi}
the operator $B:=-3A^2-A$ is sectorial in $C([-\ell_0/2,\ell_0/2])$
with domain $D(A^2)$. Moreover,
$D_{B}(\alpha,\infty)=C_{\sharp}^{4\alpha}$, with equivalence of the
corresponding norms, for any $\alpha\in (0,2)$ such that
$4\alpha\not\in\N$.

The variation of constants formula shows that any solution $\Phi\in
C^{1,4}([0,+\infty)\times [-\ell_0/2,\ell_0/2])$ to the Cauchy
problem \eqref{4order} is a fixed point of the operator $\Gamma$,
formally defined by
\begin{eqnarray*}
(\Gamma(\Phi))(\tau,\cdot)=e^{\tau B}\Phi_0
+\int_0^{\tau}e^{(\tau-s)B}(\Phi_{\eta}(s,\cdot))^2ds,\qquad\;\,\tau>0,
\end{eqnarray*}
where $\{e^{tB}\}$ denotes the semigroup generated by $B$.

Let us fix $\alpha\in (1/4,1/2)$. Theorem 7.1.2 in \cite{lunardi}
implies that $\Gamma$ has a unique fixed point $\Phi$ in
$C([0,T_0];D_B(\alpha,\infty))$. A bootstrap argument allows to
prove that $\Phi$ belongs to ${\mathcal Y}_{T_0}$. Using \cite[Prop.
4.2.1]{lunardi} and our assumptions on $\Phi_0$, it can be shown,
first that $\Phi\in C^{\beta,4\gamma}([0,T_0]\times
[-\ell_0/2,\ell_0/2])$ for any $\beta,\gamma\in (0,1)$, and, then,
that $\Phi_{\eta}\in C^{\beta}([0,T_0]\times [-\ell_0/2,\ell_0/2])$
for any $\beta\in (0,3/4)$. Moreover,
$D_{\eta}^{(j)}\Phi(\cdot,-\ell_0/2)\equiv
D_{\eta}^{(j)}\Phi(\cdot,\ell_0/2)$ for $j=0,1,2,3$. Next, applying
\cite[Thm. 4.3.1(i)]{lunardi}, we deduce that $\Phi\in
C^{1,4}([0,T]\times [-\ell_0/2,\ell_0/2])$ and is a solution to
Problem \eqref{4order}. Moreover, since $\Phi_0\in
D_B(1+(2+\alpha)/4,\infty)$, $\Phi_{\tau}$ is bounded in $[0,T_0]$
with values in $D_B(1/2+\alpha/4,\infty)$. Hence, the function
$\Phi_{\tau}$ belongs to $C^{0,2+\alpha}([0,T_0]\times
[-\ell_0/2,\ell_0/2])$. As a byproduct, $\Phi_{\eta\eta\eta\eta}$ is
in $C^{0,2+\alpha}([0,T_0]\times [-\ell_0/2,\ell_0/2])$ as well, and
$D^{(j)}_{\eta}\Phi(\cdot,-\ell_0/2)=D^{(j)}_{\eta}\Phi(\cdot,\ell_0/2)$
for $j=4,5,6$.

Using a continuation argument, we can extend $\Phi$ to a maximal
domain $[0,T)$ with a function (still denoted by $\Phi$) which
belongs to ${\mathscr Y}_{T'}$ for any $T'<T$.

The rest of the proof is devoted to show that $T=+\infty$. The main
step is an {\it a priori estimate} suggested by the proof \cite[Thm.
2.4]{tadmor}, which deals with $L^2$ regularity for the K--S
equation.
\par
{\em Step 2.} Here, we show that
\begin{equation}
\|\Phi_{\eta}(\t,\cdot)\|_2\le
e^{\frac{13}{6}\tau}\|D_{\eta}\Phi_0\|_2,\qq\;\,\tau\in [0,T).
\label{apriori-z}
\end{equation}
For this purpose, we introduce the function $v$, defined by
$v(\tau,\eta)=e^{-2\tau}\Phi_{\eta}(\tau,\eta)$ for any
$(\tau,\eta)\in [0,T)\times [-\ell_0/2,\ell_0/2]$. The smoothness of
$\Phi$ implies that $v\in C^{1,4}([0,T)\times
[-\ell_0/2,\ell_0/2])$, solves the parabolic equation
\begin{equation}
v_\tau= -3v_{\eta\eta\eta\eta}-v_{\eta\eta} - e^{2\tau}vv_\eta-2v,
\label{parab-eq}
\end{equation}
and satisfies the boundary conditions
$D_{\eta}^{(k)}v(\t,-\ell_0/2)=D_{\eta}^{(k)}v(\t,\ell_0/2)$ for any
$\t\in [0,T)$ and $k=0,1,2,3$. Multiplying both the sides of
\eqref{parab-eq} by $v(\tau,\cdot)$, integrating on
$(-\ell_0/2,\ell_0/2)$ and observing that the integral over
$(-\ell_0/2,\ell_0/2)$ of $(v(\t,\cdot))^2v_{\eta}(\t,\cdot)$
vanishes for any $\tau\in (0,T]$, we get
\begin{equation}
\frac{d}{d\t}\|v(\t,\cdot)\|_2^2
+3\|v_{\eta\eta}(\t,\cdot)\|_2^2-\|v_{\eta}(\t,\cdot)\|_2^2+2\|v(\t,\cdot)\|_2^2=0.
\qq\;\,\tau\in [0,T).
\label{energy-1}
\end{equation}
In view of the estimate
\begin{eqnarray*}
\|v_{\eta}(\t,\cdot)\|_2^2\le
\|v(\t,\cdot)\|_2\|v_{\eta\eta}(\t,\cdot)\|_2 \le
3\|v_{\eta\eta}(\t,\cdot)\|_2^2+\frac{5}{3}\|v(\tau,\cdot)\|_2^2,\qquad\;\,\tau\in
[0,T),
\end{eqnarray*}
Formula \eqref{energy-1} leads us to the inequality
\begin{eqnarray*}
\frac{d}{d\t}\|v(\t,\cdot)\|_2^2+\frac{1}{3}\|v(\t,\cdot)\|_2^2\le
0, \qq\;\,\t\in [0,T),
\end{eqnarray*}
from which Estimate \eqref{apriori-z} follows at once.
\par
{\em Step 3}. Let us consider the function $\Psi$, defined by
$\Psi(\tau,\eta)=\Phi(\tau,\eta)-\Pi(\Phi(\tau,\cdot))$ for any
$\tau\in [0,T)$ and any $\eta\in [-\ell_0/2,\ell_0/2]$, where
$\Pi(\Phi(\tau,\cdot))$ denotes the average of $\Phi(\tau,\cdot)$
over the interval $(-\ell_0/2,\ell_0/2)$. Applying Poincar\'e-Wirtinger
inequality, we get
\begin{equation}
\|\Phi(\tau,\cdot)-\Pi(\Phi(\tau,\cdot))\|_{\infty}\le\sqrt{\ell_0}e^{\frac{13}{6}\tau}\|D_{\eta}\Phi_0\|_2,
\qquad\;\,\tau\in [0,T). \label{I-P-sup}
\end{equation}

Let us now show that the function $\tau\mapsto\Pi(\Phi(\tau,\cdot))$
satisfies a similar estimate. For this purpose, we fix $T'\in
(0,T)$, $\tau\in [0,T')$, and apply the operator $\Pi$ to both the
sides of \eqref{4order}. Since $\Phi$ and its derivatives satisfy
periodic boundary conditions,
\begin{eqnarray*}
\frac{d}{d\tau}\Pi(\Phi(\tau,\cdot))= \Pi(\Phi_{\tau}(\tau,\cdot))=
-\frac{1}{2\ell_0}\Pi((\Phi_{\eta}(\tau,\cdot))^2),
\end{eqnarray*}
for any $\tau\in [0,T)$. Taking \eqref{apriori-z} into account, we
can then estimate
\begin{eqnarray*}
\left |\frac{d}{d\tau}\Pi(\Phi(\tau,\cdot))\right |
\le\frac{1}{2\ell_0}e^{\frac{13}{3}\tau}\|D_{\eta}\Phi_0\|_2^2,\qquad\;\,\tau\in
[0,T).
\end{eqnarray*}
Hence,
\begin{equation}
|\Pi(\Phi(\tau))|\le|\Pi(\Phi_0)|+\int_0^{\tau}\left
|\frac{d}{d\tau}\Pi(\Phi(\tau,\cdot))\right |d\tau\le |\Pi(\Phi_0)|
+\frac{3}{26\ell_0}\|D_{\eta}\Phi_0\|_2^2e^{\frac{13}{3}\tau},
\label{Pi-sup}
\end{equation}
for any $\tau\in [0,T)$. Estimates \eqref{I-P-sup} and
\eqref{Pi-sup} show that $\Phi$ is bounded in $[0,T)\times
[-\ell_0/2,\ell_0/2]$. Therefore, we can apply \cite[Prop.
7.2.2]{lunardi} with $X_{\alpha}=D_B(\alpha,\infty)$, which implies
that $T=+\infty$.

\end{document}